\documentclass[11pt]{amsart}
\usepackage[margin=1in]{geometry}
\usepackage{amsmath,amssymb,amsthm}
 
\usepackage{tikz}
\usepackage{ytableau}
\usepackage{mfabacus}
\usepackage{genyoungtabtikz}
\usetikzlibrary{decorations.pathreplacing,arrows.meta}
\usepackage{array,booktabs}
\usepackage[hidelinks]{hyperref}
\usepackage{wasysym}

\makeatletter
\@namedef{subjclassname@2020}{%
	\textup{2020} Mathematics Subject Classification}
\makeatother

\newtheorem{theorem}{Theorem}[section]
\newtheorem{proposition}[theorem]{Proposition}
\newtheorem{lemma}[theorem]{Lemma}
\newtheorem{corollary}[theorem]{Corollary}
 
\theoremstyle{definition}
\newtheorem{definition}[theorem]{Definition}
\newtheorem{example}[theorem]{Example}
 
\theoremstyle{remark}
\newtheorem{remark}[theorem]{Remark}

\DeclareMathOperator{\emp}{emp}
\DeclareMathOperator{\Emp}{Emp}

\newcommand{\vac}[3]{\Emp_{#3}(#1,#2)}

\numberwithin{equation}{section}

\title[Arm lengths and crystal reflection]{Divisible Arm Lengths, Crystal Reflections, and Enumeration of Newly Found Decomposition Columns}
 
\author{David J. Hemmer}
\address{Department of Mathematical Sciences, Michigan Technological University, Houghton, MI 49931, USA}
\email{djhemmer@mtu.edu}
 
\author{Pavel Turek}
\address{Representation Theory and Algebraic Combinatorics Unit, Okinawa Institute of Science and Technology, Okinawa, Japan}
\email{pavel.turek@oist.jp}
 
 
\begin{document}

\begin{abstract}
In recent work the authors determine complete columns of symmetric-group decomposition matrices in odd prime characteristic $p$ labeled by $p$-regular partitions for which every hook of length divisible by $p$ has even arm length. In the present paper we enumerate these partitions and prove that each block of $p$-weight $w$ contains precisely
\[
\binom{w+\frac{p-3}{2}}{w}
\]
such partitions. 

More generally, for any integers $d,e>1$, we study and enumerate $d$-balanced $e$-regular partitions -- partitions for which every hook of length divisible by $e$ has arm length divisible by $d$. Our first main result is that the crystal (affine) reflections preserve the $d$-balanced property for all $d,e > 1$. It follows that, for fixed $d$, $e$, and $w$, the number of $d$-balanced $e$-regular partitions in a block of $e$-weight $w$ is independent of the $e$-core. We then compute this number by working in RoCK blocks, obtaining an explicit binomial formula valid for every block.

We also investigate closely related odd sequences of partitions. Among others, we find the generating function of the number of odd sequences occurring in a block. Alongside their representation-theoretic relevance, we expect these results to be of independent combinatorial interest.
\end{abstract}
 
\maketitle
\section{Introduction}
 
Hooks and hook lengths are among the most studied objects in the combinatorics of partitions. The celebrated hook length formula of Frame, Robinson, and Thrall \cite{FRT} counts standard Young tableaux of a given shape, and its generalizations permeate algebraic combinatorics. The study of partitions constrained by their hook lengths has a long and rich history: $t$-core partitions, defined by the absence of hooks of length~$t$, connect partition theory to modular representation theory, the theory of cranks, and Ramanujan-type congruences \cite{GKS}. Simultaneous core partitions, self-conjugate cores, and related objects continue to generate a substantial body of work.
 
While hook \emph{lengths} have been studied extensively, the lengths of two constituents of a hook, the \emph{arm length} and the \emph{leg length}, have received comparatively little individual attention. Of course, the arm and leg lengths together determine the hook length, and the parity of the leg length appears in places such as the Murnaghan--Nakayama rule and the Jantzen--Schaper formula. By contrast, conditions imposed purely on arm lengths, such as divisibility or parity constraints on the arms of hooks of certain lengths, do not seem to have been investigated from a combinatorial perspective. Perhaps the only slight exception is the paper by Fayers \cite{FayersPartitionModelsCrystal}, in which crystal models are labelled by so-called \emph{arm sequences} $(A_1,A_2,\dots)$, and consist of partitions with no $ie$-hook of arm length $A_i$ (for any $i\geq 1$).
 
In this paper we begin a systematic study of such conditions. Let $d, e>1$ be integers. We consider a class of partitions, recently introduced by the second author \cite{turek2025mullineuxmapdbalancedpartitions}, in which every hook of length divisible by $e$ has arm length divisible by $d$; these are the \emph{$d$-balanced} partitions (with respect to~$e$).

Our interest in this class is partly representation-theoretic. When $e$ is a prime, $2$-balanced $e$-regular partitions label the decomposition-matrix columns of symmetric groups that are determined completely in \cite[Theorem~1.1]{hemmerturek2026}, and the question of how many such columns occur in a given block is what led to the present work. We fully answer this question here in a way that gives an algorithm to quickly identify these columns in any given block.  

The investigation of columns labeled by $d$-balanced $e$-regular partitions for $d>2$ is a subject of study in an ongoing work of Gustavsson, Law, Putignano, Speyer and the second author. It has already been shown in \cite[Theorem~1.3]{turek2025mullineuxmapdbalancedpartitions} that $d$-balanced $e$-regular partitions have representation-theoretic significance for any $d>1$: the Mullineux map (which for symmetric groups describes tensor products with sign) has a particularly simple description when applied to $d$-balanced $e$-regular partitions. The connections between $d$-balanced partitions and decomposition numbers -- computing which is a central open problem of representation theory of symmetric groups -- mark $d$-balanced partitions as objects of algebraic and combinatorial significance.

The crucial observation of this paper is that the $d$-balanced condition interacts naturally with the crystal graph on $e$-regular partitions and the abacus combinatorics of James. This interaction allows us to show that the number of $d$-balanced $e$-regular partitions in a block depends only on the $e$-weight and not on the $e$-core. We then compute these common values explicitly by passing to RoCK blocks, where the abacus structure is especially transparent. This strategy was used by Chuang, Miyachi and Tan in \cite{ChuangMiyachiTanTilings17} to count their $m$-increasing partitions in a given block. Comparing the forms of $d$-balanced $e$-regular and $0$-increasing partitions in RoCK blocks, one can use the fact that both families of partitions are preserved by crystal reflections to obtain that, surprisingly, all $d$-balanced $e$-regular partitions are $0$-increasing. As the combinatorial definition of crystal reflections can be extended to all partitions, this is true even if we omit $e$-regularity. This extended definition can be used to easily count all $d$-balanced partitions in a given block; however, we take a slightly different road to obtain this number, which allows us to explore the structure of $d$-balanced $e$-singular partitions.

There is also a second, more combinatorial, theme in the paper. Classical partition statistics often record only how many parts of a given type occur, for example, the number of odd parts. Here we refine that statistic by recording, for each odd part $\lambda_i$, the residue $\lambda_i-i \pmod e$ of the last node in row $i$ of the Young diagram of $\lambda$. This defines the \emph{odd sequence} of $\lambda$, which describes the decomposition columns labelled by $2$-balanced $e$-regular partitions in \cite{hemmerturek2026} (for $e$ a prime).
There is a striking bijection between $2$-balanced partitions and pairs of an odd sequence and an $e$-core that leads naturally to generating functions for the odd sequences that occur in a given block.

There is a `dual' version of $d$-balanced partitions: \textit{$d$-shift balanced partitions}, partitions such that all their hooks of length divisible by $e$ have arm length congruent to $-1$ modulo $d$. One can move from $d$-balanced $e$-regular partitions to $d$-shift balanced $e$-restricted partitions by applying the Mullineux map and conjugation. All our results can be easily rephrased from $d$-balanced $e$-regular to $d$-shift balanced $e$-restricted partitions, either through analogous proofs or through the duality; we comment on this in more detail at the end of Section~\ref{se:RoCK}.

Throughout this paper, we work with a fixed integer $e > 1$; all results specialize to the symmetric-group representation-theoretic setting in (prime) characteristic $p$ by taking $e=p$. For a general integer $e$, one obtains the setting for the representation theory of Hecke algebras in quantum characteristic $e$ (and an arbitrary prime characteristic $p$). This is the setting where the term 'block' comes from: while we treat blocks here as sets of partitions, their name comes from the algebraic blocks of symmetric groups -- the indecomposable summands of $KS_n$ (where $K$ is a field) considered as a bimodule over itself -- namely, if $e$ is a prime, then our blocks are the sets of partitions $\lambda$ such that the Specht module $S^{\lambda}$ is not annihilated by a given block of symmetric groups over a field of characteristic $e$. For a general integer $e>1$, one needs to use blocks of Hecke algebras instead. 

\medskip
\noindent\textbf{Main results}
Our first main theorem shows that the crystal (affine) reflections $\sigma_i$ preserve the $d$-balanced property (Theorem~\ref{thm:main}). As a consequence, the number of $d$-balanced $e$-regular partitions is the same in every block of a fixed $e$-weight (Corollary~\ref{cor:samenumbereveryblockweightw}). Our second main result computes this number explicitly: by carrying out the count in RoCK blocks, we prove that it is
\[
\binom{w + \left\lfloor\frac{e-1}{d}\right\rfloor - 1}{w}
\]
(Theorem~\ref{thm:maincountRouquier}). The proof provides an algorithmic construction of these columns: they can be explicitly constructed in RoCK blocks and then transferred to the desired blocks using the crystal reflections.

We also deduce that even the number of all $d$-balanced partitions in a given block depends only on the $e$-weight $w$ and not the $e$-core, and that it equals.
\[
\binom{w + \left\lfloor\frac{e-1}{d}\right\rfloor}{w}
\]
(Corollary~\ref{co:all d-balanced}).
In the odd-sequence setting, we further obtain generating functions for the total number of odd sequences (and some variations) occurring in a fixed block (Theorem~\ref{th:generating functions}).

\medskip
\noindent\textbf{Outline}
Section~\ref{se:partitions} consists of a summary of basic terminology and properties of partitions and James's abacus, and the (repeated) definition of $d$-balanced partitions. In Section~\ref{se:crystal} we describe the crystal (affine) reflection on partitions, both using Young diagrams and James's abacus. A reader familiar with partitions, James's abacus (and crystal reflections) should be able to only glance over Section~\ref{se:partitions} (and Section~\ref{se:crystal}). The proof of our first main result -- crystal (affine) reflections preserve the $d$-balanced property -- is the subject of Section~\ref{sec:rimhooksandarmlength}. We enumerate the $d$-balanced $e$-regular partition in Section~\ref{se:RoCK} by describing their form in RoCK blocks, and deduce the enumeration for all $d$-balanced partitions. The applications of our results to counting odd sequences are in Sections~\ref{se:odd}.

\section{Hooks, rim hooks and James's abacus}\label{se:partitions}
 
In this section we define the set of partitions we are interested in. Throughout, let $e>1$ be a fixed integer.
 
\begin{definition}
A \emph{partition} of $n$ is a weakly decreasing sequence
$\lambda = (\lambda_1 \geq \lambda_2 \geq \cdots \geq \lambda_\ell > 0)$ of positive
integers, called \textit{parts}, summing to $n$; we write $|\lambda| = n$. The \emph{Young diagram} of $\lambda$
is the set of \emph{nodes}
\[
  [\lambda] = \{(r,c) : 1 \leq r \leq \ell,\; 1 \leq c \leq \lambda_r\},
\]
drawn as a left-adjusted diagram with row $r$ containing $\lambda_r$ nodes. A partition is \emph{$e$-regular} if no
part is repeated $e$ or more times; otherwise, it is called \textit{$e$-singular}.
\end{definition}

Throughout, by partition we simply mean a partition of some integer $n\geq 0$. Occasionally, we will need a `conjugate' version of $e$-regular partition. The \textit{conjugate partition} of $\lambda$, denoted by $\lambda'$, is the partition with Young diagram $[\lambda'] = \{(r,c) : (c,r)\in [\lambda]\}$. In other words, the Young diagrams of $\lambda$ and $\lambda'$ are reflections of each other along the main diagonal. We say that $\lambda$ is \textit{$e$-restricted} if $\lambda'$ is $e$-regular. That is, $\lambda$ is $e$-restricted if its Young diagram does not contain $e$ or more columns of the same length.

We now recall definitions of hooks and rim hooks. To do so, we use an intuitive terminology for nodes. For example, the nodes lying south-east of node $(r,c)\in [\lambda]$ are the nodes $(r',c')$ with $r'\geq r$ and $c'\geq c$, among these are the nodes right of $(r,c)$ in row $r$ -- the nodes with $r'=r$ -- and the nodes below $(r,c)$ in column $c$ -- the nodes with $c'=c$. We can make all three notions strict by replacing $r'\geq r$ and/or $c'\geq c$ with their strict versions.

An example of hooks, rim hooks and all the so-far introduced terms is in Figure~\ref{fig:YD}.

\begin{figure}[ht]
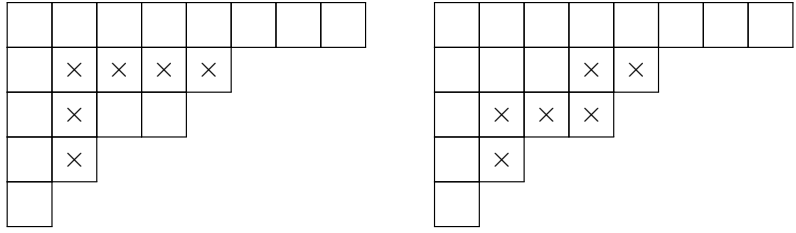

    \centering
    \begin{ytableau}
    {} &  & &  &  & & & \\
    & \times  & \times  &  \times & \times \\
    & \times & &\\
    & \times\\
    \\
    \end{ytableau}
    \qquad
    \begin{ytableau}
    {} &  & &  &  & & & \\
    & & &  \times & \times \\
    & \times & \times & \times\\
    & \times\\
    \\
    \end{ytableau}
    \caption{Let $\lambda = (8,5,4,2,1)$. In the first Young diagram of $\lambda$, the nodes with the $\times$ symbol form the hook at node $(2,2)$. In the second Young diagram, these nodes form the rim hook at the same node. The hook length, arm length and leg length of both, this hook and rim hook, are $6$, $3$ and $2$, respectively. Removing the rim hook from the second diagram yields the Young diagram of $(8,3,1,1,1)$. Since $\lambda$ has no repeated part, it is $e$-regular for all $e>1$. Since its conjugate partition $\lambda' = (5,4,3,3,2,1,1,1)$ has three copies of part $1$, $\lambda$ is neither $2$- nor $3$-restricted. But it is $e$-restricted for all $e>3$.}
    \label{fig:YD}
\end{figure}

\begin{definition}
For a node $(r,c) \in [\lambda]$, the \emph{hook} at $(r,c)$ consists
of $(r,c)$ itself, the nodes strictly to its right in row $r$, and the nodes strictly
below it in column $c$. Its \emph{hook length} is
\[
  h_{rc}(\lambda) = a_{rc}(\lambda) + \ell_{rc}(\lambda) + 1,
\]
where the \emph{arm length} $a_{rc}(\lambda) = \lambda_r - c$ counts nodes strictly to the
right of $(r,c)$ in row $r$, and the \emph{leg length}
$\ell_{rc}(\lambda) = \lambda'_c - r$ counts nodes strictly below $(r,c)$ in column $c$. 
\end{definition}

Note that the hook length is just the size of the respective hook.

\begin{definition}
For a node $(r,c) \in [\lambda]$, the \emph{rim hook} at $(r,c)$ consists of those nodes south-east of $(r,c)$ which have no node strictly south-east of them. Its \textit{hook length} is its size, \emph{arm length} is the number of columns its nodes occupy minus $1$, and its \emph{leg
length} is the number of rows its nodes occupy minus $1$. 
\end{definition}

The rim hook at $(1,1)$ is called the \textit{rim} of $[\lambda]$.
Note that one obtains a rim hook at $(r,c)\in\lambda$ by appropriately projecting nodes of the hook at $(r,c)$ south and east to the rim of $[\lambda]$. This can be used to deduce the standard fact that the hook length, arm length and leg length of the rim hook at $(r,c)$ equal the corresponding quantities of the hook at $(r,c)$. As such, hooks and rim hooks can often be used interchangeably; an advantage of hooks is their simple definition, while the advantage of rim hooks is that their removal from the Young diagram yields the Young diagram of a smaller partition.

We refer to (rim) hooks of hook length $t$ as (rim) \textit{$t$-hooks}, and to (rim) hooks of hook length divisible by $e$ as (rim) \textit{$e$-divisible hooks}.
We are interested in the following class of partitions, introduced by the second author in \cite{turek2025mullineuxmapdbalancedpartitions}.
 
\begin{definition}[{\cite[Definition~1.1]{turek2025mullineuxmapdbalancedpartitions}}]\label{def:d-balanced}
Let $d>1$ be an integer.  We say that $\lambda$ is \emph{$d$-balanced} (with respect to~$e$) if the arm length of every $e$-divisible hook of $\lambda$ is congruent to $0$ modulo~$d$.  
\end{definition}

An example of a $d$-balanced partition is in Figure~\ref{fig:d-balanced}.

\begin{figure}[ht]
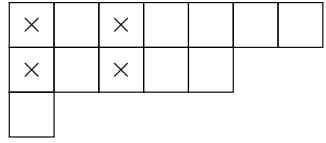

    \centering
    \begin{ytableau}
    \times &  & \times &  &  & &  \\
    \times &  & \times  &  &  \\
     \\
    \end{ytableau}
    \caption{Let $e=3$ and $d=2$. Partition $(7,5,1)$ is $2$-balanced since its four $3$-divisible hooks -- hooks at nodes $(1,1)$, $(1,3)$, $(2,1)$ and $(2,3)$, denoted by $\times$ in the diagram -- have arm lengths $6$, $4$, $4$ and $2$, respectively, all divisible by $2$.}
    \label{fig:d-balanced}
\end{figure}

\subsection{Abacus notation}
Our proofs will heavily use the abacus representation of partitions on~$e$ runners. We introduce it here following the exposition from \cite[Section~2.7]{JamesKerber} (with a slight exception of using infinitely many beads).
 
\begin{definition}
Let $\lambda=(\lambda_1,\lambda_2,\dots, \lambda_t)$ be a partition. Fix an integer $r$.
The \emph{$\beta$-set} of $\lambda$ (with respect to $r$) is the set
\[
B_r(\lambda)
=
\{\lambda_i-i+r \mid i \in \mathbb{N}\}
\]
where we take $\lambda_i$ to be zero for $i>t$. We set  $\beta_i=\lambda_i-i+r$, so that $\beta_1>\beta_2>\dots$. The partition may be recovered from the $\beta$-set by $\lambda_i=\beta_i-(r-i)$.
  
The \emph{$e$-abacus} of $\lambda$ (with respect to $r$) has runners labelled $0, 1, \ldots, e-1$ and drawn vertically.
Runner $i$ contains positions $\ldots, i-e,i,\, i+e,\, i+2e,\, \ldots$ from top to bottom; position $ke + i$ is at \emph{level} $k$ on runner $i$. A position is \emph{occupied} (a bead) if it belongs to
$B_r(\lambda)$, and \emph{vacant} (a gap) otherwise.
\end{definition}

\begin{remark}\label{re:beta-set}
We mention some standard properties of the $e$-abacus of a partition.
\begin{enumerate}
    \item When $r$ is the number of parts of $\lambda$, the corresponding $\beta$ set is the set $\{h_{j1}(\lambda)\}$ of first-column hook lengths (which we take to be $0$ if $j>r$).
    \item Changing $r$ by a multiple of $e$ slides all runners simultaneously and does not affect the combinatorial invariants defined below.
    \item Notice every $\beta$-set contains $\mathbb{Z}\cap(-\infty,u)$ for some $u$, so each runner has infinitely many beads at the top (that is, at all sufficiently small levels), which we do not display. With a slight exception in Section~\ref{se:RoCK}, henceforth we assume that our abaci have a multiple of $e$ displayed beads. This will usually be done by taking $r=0$; we refer to the corresponding $\beta$-set and $e$-abacus of $\lambda$ as \textit{canonical}.
    \item There is a quick way to find $\lambda$ from its $e$-abacus. For each occupied position, we count the number of vacant lower-numbered positions. Arranging the non-zero entries in a weakly decreasing order recovers $\lambda$. As such, $\lambda$ is $e$-regular if and only if it does not contain beads at $e$ consecutive positions and a gap at a lower-numbered position.
\end{enumerate}
\end{remark}

The positions of the $e$-abacus are arranged as follows:

\[
\begin{array}{cccc}
\vdots&\vdots&&\vdots\\
0 & 1 & \cdots & e-1 \\
e & e+1 & \cdots & 2e-1 \\
2e & 2e+1 & \cdots & 3e-1 \\
\vdots & \vdots & & \vdots
\end{array}
\]
The displayed levels are levels $0$, $1$ and $2$ (from top to bottom), and the displayed runners are $0$, $1$ and $e-1$ (from left to right).

Given a bead on runner $i$ at level $k$ and a gap on the same runner at level $k-1$, we can \textit{slide up} this bead by changing the levels $k$ and $k-1$ on runner $i$ to a gap and a bead, respectively. Formally, we replace $ke+i$ in the underlying $\beta$-set by $ke+i-e$. If we start with an $e$-abacus of a partition of $n$, sliding a bead up produces a $e$-abacus of a partition of $n-e$ (with respect to the same integer $r$) (compare with Lemma~\ref{lem:hook}). The inverse process, \textit{sliding down} a bead, enjoys similar properties.

It is clear that changing the underlying integer $r$ commutes with sliding beads up, more precisely, whether we can get from an $e$-abacus of a partition $\lambda$ with respect to $r$ to an $e$-abacus of a partition $\mu$ with respect to $r$ by sliding a bead up is independent of the choice of $r$. This justifies the following definition. In it (and throughout the rest of the paper), we may identify $\lambda$ with its $e$-abacus, if there is no risk of confusion.

\begin{definition}\label{def:core}
The \emph{$e$-core} of $\lambda$ is obtained by sliding the beads of $\lambda$ up as much as possible (i.e. until there is no more bead to slide up). The \emph{$e$-weight} of $\lambda$ is $(|\lambda| - |\gamma|)/e$, where $\gamma$ is the $e$-core of $\lambda$. A \emph{block} is a set of partitions with given $e$-core and $e$-weight (called the $e$-core and $e$-weight of the block).
\end{definition}

\begin{remark}\label{re:rim hook removal}
    We see later in Lemma~\ref{lem:hook} that the $e$-core of $\lambda$ can be obtained by removing rim $e$-hooks in the Young diagram of $\lambda$ until there are none present. In particular, if $\mu$ is obtained from $\lambda$ by removing a rim $e$-hook, then they have the same $e$-core and, using the above definition, the $e$-weight of $\mu$ is one less than the $e$-weight of $\lambda$. We will use these facts in Section~\ref{se:odd}.
\end{remark}

Since sliding a bead up decreases the size of a partition by $e$, one can think of the $e$-weight of $\lambda$ as the number of slides made when getting from $\lambda$ to its $e$-core. It is therefore clear that sliding a bead up preserves the $e$-core and decreases the $e$-weight by one.
 
\begin{example}\label{ex:abacus}
We illustrate Definition~\ref{def:core} with $e = 3$ and $\lambda=(7,5,4,3,3,2,1)$. Choosing $r=9$ we get the $\beta$-set $B_9(\lambda)=\{\cdots, -2,-1,0,1,3,5,7,8,10,12,15\}.$ On the abacus we represent $\lambda$ with 9 beads as:
 
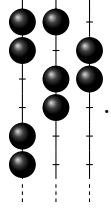
\begin{center}
\abacus(lmr,bbn,bnb,nbb,nbn,bnn,bnn,vvv).
\end{center}
Sliding the beads up, we get the abacus representation for the $3$-core, which represents the partition $(1)$:
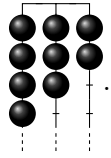
\begin{center}
\abacus(lmr,bbb,bbb,bbn,bnn,vvv).
\end{center}
As a total of $8$ slides were made, the $e$-weight of $\lambda$ is $8$.
\end{example}
 
\section{Crystal reflections}\label{se:crystal}
The number of $e$-regular partitions in a block of $e$-weight $w$ does not depend on the $e$-core, and equals the number of $(e-1)$-tuples of partitions whose sizes add up to $w$ \cite[Theorem 6.2.2]{JamesKerber}. In this section we recall the well-known crystal operators, which give bijections between $e$-regular partitions in blocks of the same $e$-weight with different $e$-cores. An alternative formulation does the same for $e$-restricted partitions in a block. The original formulation of this is in \cite[Theorem~4.7]{misramiwa}, see also \cite{FayersPartitionModelsCrystal} for the $e$-restricted version. The combinatorial definition of these crystal operators can be generalized to all partitions; however, we stick with the definition for $e$-regular partitions, and use this self-imposed restriction to study $d$-balanced $e$-singular partitions in more detail when counting all $d$-balanced partitions in a block (see Lemma~\ref{le:add a column}).
 
The definitions below are taken from \cite[Chapter~11]{kleshchev}. The \textit{$e$-residue} of node $(r, c)$ is the remainder of $(c - r)$ modulo $e$ (between $0$ and $e-1$). Here and in the rest of the paper, a node is any pair of positive integers (not necessarily an element of a Young diagram).
 
\begin{definition}
Let $\lambda$ be a partition and $i\in \{0, 1, \dots, e-1\}$. A node of $e$-residue $i$ is
\emph{$i$-removable} (respectively \emph{$i$-addable}) if removing it from (respectively adding it to)
$[\lambda]$ gives a Young diagram of a partition.
\end{definition}

\begin{definition}
Let $\lambda$ be a partition. Label all its $i$-addable nodes by $+$ and its $i$-removable nodes by $-$. Then the \emph{$i$-signature} of $\lambda$ is the sequence of pluses and minuses obtained by going along the rim of the Young diagram from bottom left to upper right and reading off the signs. The \emph{reduced $i$-signature} is obtained by successively deleting all neighboring pairs of the form $-+$ until none remain.
\end{definition}

Since any $-$ immediately followed by $+$ is cancelled, the reduced $i$-signature
always has the form
\[
  \underbrace{+\cdots+}_{a} \underbrace{-\cdots-}_{b}, \qquad a,b \geq 0.
\]
In Figure~\ref{fig:isignature} we determine all (reduced) $i$-signatures for 
$\lambda=(14,11,10,7,5,4,1)$ and $e=3$. The relevant $3$-residues are labelled.
 
\begin{figure}[ht]
    \centering
    \Yinternals0\Yaddables1
    $\yngres(3,14,11,10,7,5,4,1)$
    \caption{Unreduced $1$-signature (bottom to top)
    ${\color{red}+}{-}{+}{\color{red}+}{-}{+}{\color{red}-}$,
     reduces to ${\color{red}+}{\color{red}+}{\color{red}-}$. The $2$-signature $++++$ and $0$-signature $----$ are already reduced.}
    \label{fig:isignature}
\end{figure}
 
Removable nodes corresponding to a ``$-$'' in the reduced $i$-signature are called 
\emph{$i$-normal}, and addable nodes corresponding to a ``$+$'' are called 
\emph{$i$-conormal}. 
 
In Kleshchev's language~\cite{kleshchev}, if it exists, the $i$-normal node corresponding to the leftmost $-$ is called 
\emph{$i$-good} and the $i$-conormal node corresponding to the rightmost $+$ is called \emph{$i$-cogood}, or \emph{good addable}.
In Figure~\ref{fig:isignature} the reduced $1$-signature is highlighted in red (i.e. it is given by the first, fourth and the last symbol of the $1$-signature). The node $(1,14)$ is $1$-good and the node $(4,8)$ is good addable.
 
For an $i$-removable node $A$ in a partition $\lambda$ we let $\lambda_A$ be $\lambda$ with the node $A$ removed and similarly for $B$ an $i$-addable node we let $\lambda^B$ be $\lambda$ with the node $B$ added (identifying partitions with their Young diagrams throughout).
 
We are now ready to define the crystal reflections \cite[11.1]{kleshchev}. Let $\varepsilon_i(\lambda)$ be the number of $i$-normal nodes in $[\lambda]$ and $\varphi_i(\lambda)$ the number of $i$-conormal nodes. These are also the number of minuses and pluses, respectively, in the reduced $i$-signature of $\lambda$.  Let 
 
\begin{align*}
\tilde{e}_i(\lambda) &= \begin{cases}
    \lambda_A & \text{if } \varepsilon_i(\lambda) > 0 \text{ and } A \text{ is the (unique) } i\text{-good node,} \\
    0        & \text{if } \varepsilon_i(\lambda) = 0,
\end{cases} \\[10pt]
\tilde{f}_i(\lambda) &= \begin{cases}
    \lambda^B & \text{if } \varphi_i(\lambda) > 0 \text{ and } B \text{ is the (unique) } i\text{-cogood node,} \\
    0        & \text{if } \varphi_i(\lambda) = 0.
\end{cases}
\end{align*}
 
So for our example above we see that $$\tilde{e}_1(14,11,10,7,5,4,1)=(13,11,10,7,5,4,1)$$ and $$\tilde{f}_1(14,11,10,7,5,4,1)=(14,11,10,8,5,4,1).$$
 
\begin{definition}
\label{def:reflections}
Let $\lambda$ be $e$-regular and $i \in \{0,1,\ldots, e-1\}$. Define the \textit{affine $i$-reflection}:
\begin{align*}
\sigma_i(\lambda) &= \begin{cases}
    \tilde{f}_i^{\varphi_i(\lambda)-\varepsilon_i(\lambda)}(\lambda) & \text{if } \varepsilon_i(\lambda)< \varphi_i(\lambda),\\
  \tilde{e}_i^{\varepsilon_i(\lambda)-\varphi_i(\lambda)}(\lambda) & \text{if } \varepsilon_i(\lambda)> \varphi_i(\lambda),\\
  \lambda & \text{if } \varepsilon_i(\lambda) = \varphi_i(\lambda).
\end{cases} \\[10pt]
\end{align*}
\end{definition}
 
In short, in the first case, $\sigma_i$ adds nodes corresponding to the $\varphi_i(\lambda)-\varepsilon_i(\lambda)$ rightmost $+$'s in the reduced $i$-signature, in the second case it removes nodes corresponding to the $\varepsilon_i(\lambda)-\varphi_i(\lambda)$ leftmost $-$'s in the reduced $i$-signature, and it does nothing in the third case. The following is well known, see for example \cite{misramiwa} or Kleshchev's book \cite{kleshchev}:

\begin{proposition}\label{prop:bijection}
The affine $i$-reflection $\sigma_i$ gives an involutive bijection between $e$-regular partitions
of $e$-weight $w$ in the block with $e$-core $\gamma$ and those in the block with $e$-core
$\sigma_i(\gamma)$, where $\sigma_i(\gamma)$ is obtained from $\gamma$ by swapping runners
$i-1$ and $i$ in its $e$-abacus representation.
\end{proposition}
 
By applying appropriate sequences of these reflections (commonly called \textit{crystal affine reflections}) we can move from one $e$-core to any other:
 
\begin{proposition}\label{prop:crystal-bijection}
Let $B$ and $B'$ be two blocks of common $e$-weight~$w$. There is a bijection between the sets of $e$-regular partitions in $B$ and $B'$, given by a suitable composition of crystal affine reflections.
\end{proposition}
 
\subsection{Crystal reflections on the abacus}
 
Recall that we represent partitions on the abacus with a multiple of $e$ beads (that is, using an $e$-abacus with respect to $r$, divisible by $e$). In this case, the $i$-signature is determined entirely by runners $i-1$ and $i$ (taken modulo~$e$). At each level $k$, the occupancy of positions $ke+(i-1)$ and $ke+i$ contributes
as follows.
 
\begin{center}
\renewcommand{\arraystretch}{1.2}
\begin{tabular}{ccc}
  \toprule
  Runner $i-1$ & Runner $i$ & Contribution \\
  \midrule
  bead   & vacant & $+$ \\
  vacant & bead   & $-$ \\
  bead   & bead   & (none) \\
  vacant & vacant & (none) \\
  \bottomrule
\end{tabular}
\end{center}
 
Reading in order of increasing level gives the $i$-signature. We see that the levels at which both runners
agree are `invisible' to the signature and play no role in the cancellation. If $i=0$, then one has to compare levels $k$ on runner $0$ with levels $k-1$ on runner $e-1$; to unify our arguments, for $i=0$, we let the level $k$ on runner $i-1=-1$ be the level $k-1$ on runner $e-1$ (which is the position $ke+(i-1)$).
 
\begin{example}\label{ex:isig}
Continuing with $\lambda$ in Example~\ref{ex:abacus} with $e = 3$ and $i = 2$ (so we look at runners $1$ and $2$), the contributions level by level are:
 
\begin{center}
\renewcommand{\arraystretch}{1.2}
\begin{tabular}{cccc}
  \toprule
  Level & Runner $1$ & Runner $2$ & Contribution \\
  \midrule
  $0$ & bead   & vacant & $+$ \\
  $1$ & vacant & bead   & $-$ \\
  $2$ & bead   & bead   & (none) \\
  $3$ & bead   & vacant & $+$ \\
   $4$ & vacant   & vacant & (none) \\
    $5$ & vacant   & vacant & (none) \\
  \bottomrule
\end{tabular}
\end{center}
 
The unreduced $2$-signature is $+\,-\,+$. The $-$ at level~$1$ and $+$ at level~$3$
cancel, leaving reduced $2$-signature as $+$ (from level~$0$), so $\varepsilon_i(\lambda) = 0$ and $\varphi_i(\lambda) = 1$.
\end{example}

To describe the effect of $\sigma_i$ using an abacus, if level $k$ on runners $i-1$ and $i$ is a gap and a bead, respectively, we define \textit{left move} of the bead at $ke+i$ as making level $k$ on runners $i-1$ and $i$ a bead and a gap, respectively (that is, replacing $ke+i$ by $ke+i-1$ in the underlying $\beta$-set); this creates an $e$-abacus of a new partitions (with respect to the same integer $r$). Reversing the process, we obtain the \textit{right move} of the bead at $ke+i-1$. When implicit from the context, we may omit the position or/and the words `left' and `right'.

The description of $\sigma_i$ is now easy: if $\varepsilon_i(\lambda) < \varphi_i(\lambda)$, we move right the beads at runner $i-1$ at the $\varphi_i(\lambda) - \varepsilon_i(\lambda)$ least levels contributing by $+$ to the reduced $i$-signature, if $\varepsilon_i(\lambda) > \varphi_i(\lambda)$, we move left the beads at runner $i$ at the $\varepsilon_i(\lambda) - \varphi_i(\lambda)$ greatest levels contributing by $-$ to the reduced $i$-signature. (And we do nothing if $\varphi_i(\lambda) = \varepsilon_i(\lambda)$.)

\begin{example} 
Continuing Example~\ref{ex:isig}, the abacus on the left below shows $\lambda$ with the bead to be moved (runner~$1$,
level~$0$) marked with an open circle; the abacus on the right shows $\sigma_2(\lambda)$
after the affine $2$-reflection.
 
\begin{center}
\begin{minipage}[t]{0.3\textwidth}
  \centering
  \abacus(lmr,bon,bnb,nbb,nbn,bnn,bnn,vvv)
 
  \smallskip
  {\small $\lambda$: reduced $2$-signature is $+$}
\end{minipage}
\hspace{2.5cm}
\begin{minipage}[t]{0.3\textwidth}
  \centering
 \abacus(lmr,bno,bnb,nbb,nbn,bnn,bnn,vvv)
 
  \smallskip
  {\small $\sigma_2(\lambda)$: reduced $2$-signature is $-$}
\end{minipage}
\end{center}
 
In $\sigma_2(\lambda)$, the new $2$-signature is $-\,-\,+$, and after cancellation of the
$-\,+$ at levels $1$ and $2$, the reduced signature is $-$ (from level~$0$), giving
$\varepsilon_i(\sigma_i(\lambda)) = 1$ and $\varphi_i(\sigma_(\lambda)) = 0$; thus $\sigma_i^2(\lambda)=\lambda$, as expected.
\end{example}
 
\section{Preservation of the $d$-balanced property}
\label{sec:rimhooksandarmlength}
We are now ready to prove that the crystal reflections preserve the $d$-balanced property. We will use the abacus description of the crystal reflections. To do so we need to recognize $e$-divisible hooks and their arm lengths from the abacus; this is easy to derive (see \cite[Section~2.7]{JamesKerber} for a standard reference). Henceforth we write $\vac{u}{v}{\lambda}$ for the number of vacant positions strictly between positions $u$ and $v$ in the canonical $e$-abacus representation of $\lambda$ (this is different from the conventions for $\emp_B(u,v)$ in \cite{hemmerturek2026, turek2025mullineuxmapdbalancedpartitions}).
 
\begin{lemma}\label{lem:hook}
There is a bijection between $e$-divisible hooks of $\lambda$ and pairs of a bead and a gap on the same runner of the canonical $e$-abacus of a partition $\lambda$, with the gap lying on a lower level than the bead. If they lie on runner $i$, $k$ is the level of the bead, and $k'<k$ is the level of the gap, then the corresponding hook and arm lengths are $(k-k')e$ and $\vac{k'e+i}{ke+i}{\lambda}$, respectively. Moreover, replacing this bead by a gap and this gap by a bead corresponds to the removal of the corresponding rim $e$-divisible hook.
\end{lemma}

\begin{remark}\label{re:Other r}
    There is nothing special about the canonical $e$-abacus; one can fix any integer $r$ and state the same results for $e$-abaci with respect to $r$.
\end{remark}

With $k,k'$ and $i$ as in Lemma~\ref{lem:hook}, we say that the corresponding $e$-divisible hook is \textit{on runner $i$} and goes \textit{from level $k$ to level $k'$}.
 
In Figure \ref{figure:samplereflectionabacus} we give two examples of what runners $i-1$ and $i$ may look like before applying the reflection $\sigma_i$. On the left, the $i$-signature is  $+,+,-,+,+,-,-,+$ and the reduced $i$-signature is ${\color{red}+\,+\,+\,-}$. Two beads (levels $2$ and $5$, open circles) move right under $\sigma_i$ (from runner $i{-}1$ to runner $i$), reflecting the addition of two $i$-conormal nodes to the Young diagram. On the right, the $i$-signature is $+,-,+,-,-,-,-,+$ and the reduced $i$-signature is ${\color{red}+\,-\,-\,-}$. Two beads (levels $4$ and $5$, open circles) move left under $\sigma_i$ (from runner $i$ to runner $i{-}1$), reflecting the removal of two $i$-normal nodes from the Young diagram.

\begin{figure}[ht]
\centering
\begin{minipage}[t]{0.44\textwidth}
\centering
\begin{tikzpicture}[>=Stealth, scale=1.0]
 
  \def\rs{0.72}
  \def\rw{1.9}
  \def\br{0.14}
  \def\dl{0.11}
  \def\rx{\rw+0.58}
 
  \draw[thick] (0,   0.1) -- (0,   -8.3*\rs);
  \draw[thick] (\rw, 0.1) -- (\rw, -8.3*\rs);
 
  \node[font=\small] at (0,    1.05*\rs) {$i{-}1$};
  \node[font=\small] at (\rw,  1.05*\rs) {$i$};
 
  \node at (0,    0.55*\rs)  {$\vdots$};
  \node at (\rw,  0.55*\rs)  {$\vdots$};
  \node at (0,   -8.75*\rs)  {$\vdots$};
  \node at (\rw, -8.75*\rs)  {$\vdots$};
 
  \node[font=\scriptsize\itshape] at (1.2*\rx, 0.55*\rs) {$i$-signature};

  \draw (-\dl, 0) -- (\dl, 0);
  \draw (\rw-\dl, 0) -- (\rw+\dl, 0);

  \fill (0, -\rs) circle (\br);
  \draw (\rw-\dl, -\rs) -- (\rw+\dl, -\rs);
  \node[font=\small, red] at (\rx, -\rs) {$+$};

  \fill[white] (0, -2*\rs) circle (\br);
  \draw        (0, -2*\rs) circle (\br);
  \draw (\rw-\dl, -2*\rs) -- (\rw+\dl, -2*\rs);
  \draw[->, thick, blue!75!black]
    (\br+0.04, -2*\rs) -- (\rw-\br-0.04, -2*\rs);
  \node[font=\small, red] at (\rx, -2*\rs) {$+$};

  \draw (-\dl, -3*\rs) -- (\dl, -3*\rs);
  \fill (\rw, -3*\rs) circle (\br);
  \node[font=\small] at (\rx, -3*\rs) {$-$};

  \fill (0, -4*\rs) circle (\br);
  \draw (\rw-\dl, -4*\rs) -- (\rw+\dl, -4*\rs);
  \node[font=\small] at (\rx, -4*\rs) {$+$};
 
  \draw[decorate, decoration={brace, amplitude=4pt}]
    (\rw+0.95, -3*\rs+0.06) -- (\rw+0.95, -4*\rs-0.06)
    node[midway, right=5pt, font=\scriptsize] {cancel};

  \fill[white] (0, -5*\rs) circle (\br);
  \draw        (0, -5*\rs) circle (\br);
  \draw (\rw-\dl, -5*\rs) -- (\rw+\dl, -5*\rs);
  \draw[->, thick, blue!75!black]
    (\br+0.04, -5*\rs) -- (\rw-\br-0.04, -5*\rs);
  \node[font=\small, red] at (\rx, -5*\rs) {$+$};

  \draw (-\dl, -6*\rs) -- (\dl, -6*\rs);
  \fill (\rw, -6*\rs) circle (\br);
  \node[font=\small, red] at (\rx, -6*\rs) {$-$};

  \draw (-\dl, -7*\rs) -- (\dl, -7*\rs);
  \fill (\rw, -7*\rs) circle (\br);
  \node[font=\small] at (\rx, -7*\rs) {$-$};

  \fill (0, -8*\rs) circle (\br);
  \draw (\rw-\dl, -8*\rs) -- (\rw+\dl, -8*\rs);
  \node[font=\small] at (\rx, -8*\rs) {$+$};
 
  \draw[decorate, decoration={brace, amplitude=4pt}]
    (\rw+0.95, -7*\rs+0.06) -- (\rw+0.95, -8*\rs-0.06)
    node[midway, right=5pt, font=\scriptsize] {cancel};
 
\end{tikzpicture}
 
\smallskip
$\varepsilon_i(\lambda)=1,\;\varphi_i(\lambda)=3$: reduced $i$-signature\ ${\color{red}+\,+\,+\,-}$
\end{minipage}
\hfill
\begin{minipage}[t]{0.44\textwidth}
\centering
\begin{tikzpicture}[>=Stealth, scale=1.0]
 
  \def\rs{0.72}
  \def\rw{1.9}
  \def\br{0.14}
  \def\dl{0.11}
  \def\rx{\rw+0.58}
 
  \draw[thick] (0,   0.1) -- (0,   -8.3*\rs);
  \draw[thick] (\rw, 0.1) -- (\rw, -8.3*\rs);
 
  \node[font=\small] at (0,    1.05*\rs) {$i{-}1$};
  \node[font=\small] at (\rw,  1.05*\rs) {$i$};
 
  \node at (0,    0.55*\rs)  {$\vdots$};
  \node at (\rw,  0.55*\rs)  {$\vdots$};
  \node at (0,   -8.75*\rs)  {$\vdots$};
  \node at (\rw, -8.75*\rs)  {$\vdots$};
 
  \node[font=\scriptsize\itshape] at (1.2*\rx, 0.55*\rs) {$i$-signature};

  \fill (0,   0) circle (\br);
  \fill (\rw, 0) circle (\br);

  \fill (0, -\rs) circle (\br);
  \draw (\rw-\dl, -\rs) -- (\rw+\dl, -\rs);
  \node[font=\small, red] at (\rx, -\rs) {$+$};

  \draw (-\dl, -2*\rs) -- (\dl, -2*\rs);
  \fill (\rw, -2*\rs) circle (\br);
  \node[font=\small] at (\rx, -2*\rs) {$-$};

  \fill (0, -3*\rs) circle (\br);
  \draw (\rw-\dl, -3*\rs) -- (\rw+\dl, -3*\rs);
  \node[font=\small] at (\rx, -3*\rs) {$+$};
 
  \draw[decorate, decoration={brace, amplitude=4pt}]
    (\rw+0.95, -2*\rs+0.06) -- (\rw+0.95, -3*\rs-0.06)
    node[midway, right=5pt, font=\scriptsize] {cancel};

  \draw (-\dl, -4*\rs) -- (\dl, -4*\rs);
  \fill[white] (\rw, -4*\rs) circle (\br);
  \draw        (\rw, -4*\rs) circle (\br);
  \draw[<-, thick, blue!75!black]
    (\br+0.04, -4*\rs) -- (\rw-\br-0.04, -4*\rs);
  \node[font=\small, red] at (\rx, -4*\rs) {$-$};

  \draw (-\dl, -5*\rs) -- (\dl, -5*\rs);
  \fill[white] (\rw, -5*\rs) circle (\br);
  \draw        (\rw, -5*\rs) circle (\br);
  \draw[<-, thick, blue!75!black]
    (\br+0.04, -5*\rs) -- (\rw-\br-0.04, -5*\rs);
  \node[font=\small, red] at (\rx, -5*\rs) {$-$};

  \draw (-\dl, -6*\rs) -- (\dl, -6*\rs);
  \fill (\rw, -6*\rs) circle (\br);
  \node[font=\small, red] at (\rx, -6*\rs) {$-$};
  
  \draw (-\dl, -7*\rs) -- (\dl, -7*\rs);
  \fill (\rw, -7*\rs) circle (\br);
  \node[font=\small] at (\rx, -7*\rs) {$-$};
 
  \fill (0, -8*\rs) circle (\br);
  \draw (\rw-\dl, -8*\rs) -- (\rw+\dl, -8*\rs);
  \node[font=\small] at (\rx, -8*\rs) {$+$};
 
  \draw[decorate, decoration={brace, amplitude=4pt}]
    (\rw+0.95, -7*\rs+0.06) -- (\rw+0.95, -8*\rs-0.06)
    node[midway, right=5pt, font=\scriptsize] {cancel};
 
\end{tikzpicture}
 
\smallskip
$\varepsilon_i(\lambda)=3,\;\varphi_i(\lambda)=1$: reduced $i$-signature\ ${\color{red}+\,{-}\,{-}\,{-}}$
\end{minipage}
 
\bigskip
\begin{minipage}{\textwidth}
\centering
\begin{tikzpicture}[>=Stealth]
  \def\br{0.13}
  \fill[white](0,0) circle(\br); \draw(0,0) circle(\br);
  \draw[->, blue!75!black](0.18,0)--(0.55,0);
  \node[right, font=\scriptsize] at (0.6,0) {bead moves under $\sigma_i$};
  \fill(5.0,0) circle(\br);
  \node[right, font=\scriptsize] at (5.18,0) {bead (stays)};
  \draw(-\br,-0.45)--(\br,-0.45);
  \node[right, font=\scriptsize] at (0.18,-0.45) {vacant position};
  \node[red, font=\scriptsize] at (5.0,-0.45) {\textbullet};
  \node[right, font=\scriptsize, red] at (5.18,-0.45) {reduced $i$-signature};
\end{tikzpicture}
\end{minipage}
 
\caption{Abacus diagrams illustrating $\sigma_i$ for $\varepsilon_i(\lambda)<\varphi_i(\lambda)$ (left) and $\varepsilon_i(\lambda)>\varphi_i(\lambda)$ (right). }
\label{figure:samplereflectionabacus}
\end{figure}
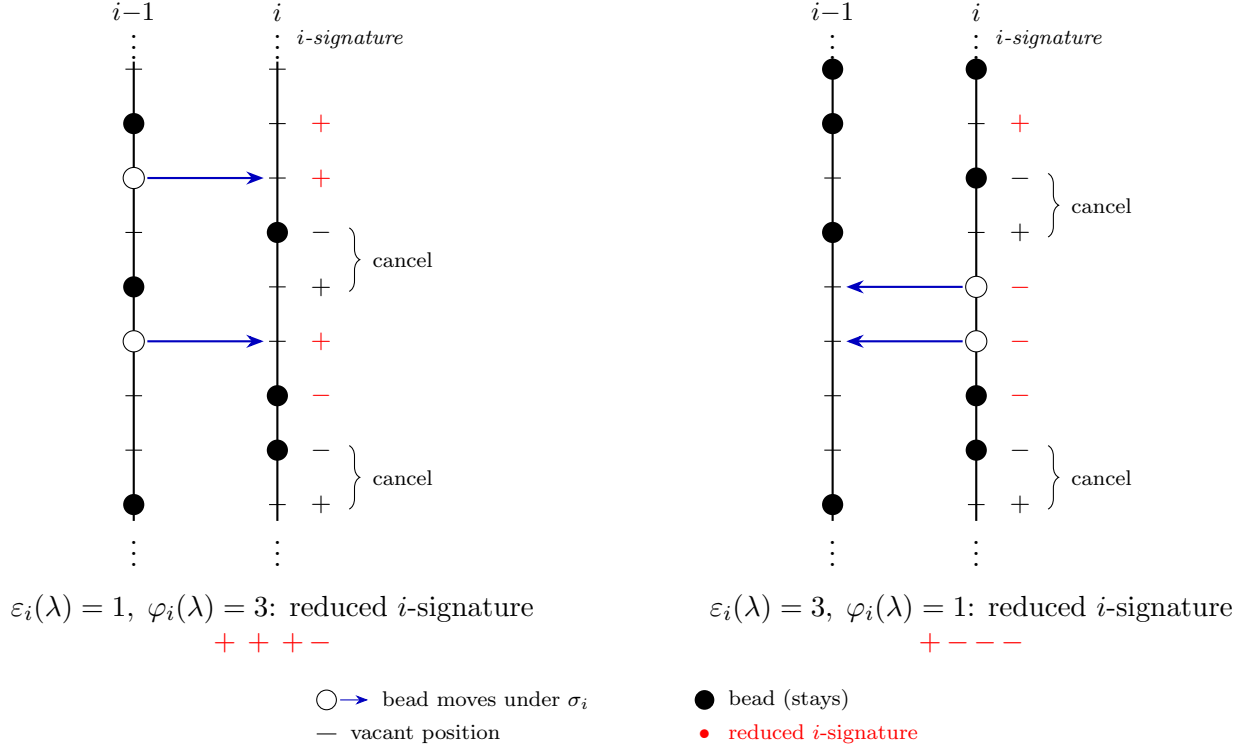
 
In either case, using the correspondence from Lemma~\ref{lem:hook}, `new' $e$-divisible hooks are created in two different ways. Beads that moved may have empty spaces at lower-numbered positions on their new runner. They also leave vacant positions that may be paired with beads at higher numbered positions on their original runner. These new hooks will be the focus of the proof that the crystal reflections preserve the $d$-balanced property: we assume all the $e$-divisible hooks in $\lambda$ have arm lengths divisible by $d$ and prove that this is the case as well for $\sigma_i(\lambda)$, by showing the newly created hooks have this property and that the existing $e$-divisible hooks keep the property.
 
We often need to understand how $\emp_{\lambda}$ changes when we change its arguments by $1$. As this is simple, we only provide a detailed proof of this change in one particular case.
 
\begin{lemma}\label{lem:shift}
Let $\lambda$ be a partition and let $i\in\{ 0,1,\dots, e-1 \}$. Suppose that position
$k' e + i$ is a gap and position $k e + (i-1)$ is a bead in $\lambda$, where
$k' < k$. Then
\[
  \vac{k' e + i}{k e + i}{\lambda}
  = \vac{k' e + (i-1)}{k e + (i-1)}{\lambda} - 1.
\]
\end{lemma}
 
\begin{proof}
The two intervals in which we count vacant positions are $(k' e+i,\,k e+i) = \{k' e+i+1,\ldots,k e+i-1\}$
and $(k' e+(i-1),\,k e+(i-1)) = \{k' e+i,\ldots,k e+i-2\}$.
They contain the same vacant positions except for position $k' e+i$, which is
vacant by hypothesis and is included in the second interval but not the first.
(The position $k e+(i-1)$, which is in the first interval but not the second, is
a bead by hypothesis and so does not contribute to either count.)
\end{proof}

We are now ready to prove that the crystal reflections preserve the $d$-balanced property. The argument is entirely abacus-theoretic; it boils down to using Lemma \ref{lem:hook} and keeping track of the vacant positions. Figure~\ref{figure:samplereflectionabacus} can be used to display the various cases in the proof. We refer to levels where runners $i-1$ and $i$ either both have beads or are both vacant as \textit{neutral}. If the beads at level $k$ correspond to a $+$ or $-$ in the $i$-signature that is cancelled, we say it is \textit{matched}. If it corresponds to a $+$ or $-$ that remains in the reduced $i$-signature, we say it is \textit{unmatched}.

\begin{theorem}\label{thm:main}
Let $\lambda$ be a $d$-balanced partition and let $i\in \{ 0,1,\dots, e-1 \}$. Then
$\sigma_i(\lambda)$ is also $d$-balanced. 
\end{theorem}
 
\begin{proof}
Throughout the proof, all hooks are assumed to be $e$-divisible. Let $b=\varepsilon_i(\lambda)$ and $a=\varphi_i(\lambda)$, so the reduced $i$-signature of $\lambda$ has the form $+^a -^b$ and the reflection
$\sigma_i$ moves $|a-b|$ beads between runners $i-1$ and $i$ simultaneously. We can assume that $a\neq b$.
 
For $j\neq i-1, i$, the hooks on runner $j$ are unaffected by $\sigma_i$ -- we have $\vac{k'e+j}{ke+j}{\sigma_i(\lambda)} = \vac{k'e+j}{ke+j}{\lambda}$ for $k'<k$, and the pair of $ke+j$ and $k'e+j$ corresponds to a hook in either both or none, $\sigma_i(\lambda)$ and $\lambda$. Similarly, for $j=i-1,i$, we have $\vac{k'e+j}{ke+j}{\sigma_i(\lambda)} = \vac{k'e+j}{ke+j}{\lambda}$ whenever $k'<k$ and the pair $ke+j$ and $k'e+j$ correspond to a hook in both $\lambda$ and $\sigma_i(\lambda)$; thus we need only to check the $d$-balanced divisibility condition for the new hooks, that is, pairs of levels on runner $i-1$ or $i$ corresponding to a hook in $\sigma_i(\lambda)$ but not in $\lambda$.

The argument splits into two cases, depending on whether
$\sigma_i$ is adding nodes ($a>b$) or removing nodes ($b>a$), and then into subcases
depending on which types of new hooks are created (see the paragraph before Lemma~\ref{lem:shift}). Cases 1 and 2 are mirror images of each other; within each, sub-cases (a) and (b) handle whether the 'other end' of the new hook lies at a neutral level or not.
 
\bigskip
\noindent\textbf{Case 1 $a > b$: beads move right, from runner $i-1$ to runner $i$.}
 
Here $\sigma_i(\lambda) = \tilde{f}_i^{a-b}(\lambda)$, moving $a-b$ beads of $\lambda$ from runner $i-1$ to runner
$i$ at levels $S$, the $a-b$ greatest unmatched $+$'s in the reduced $i$-signature.
At each $k \in S$: runner $i-1$ has a bead and runner $i$ is
vacant in $\lambda$; after the moves, runner $i-1$ is vacant and runner $i$ has a bead.
The $b$ unmatched $-$'s appear at levels greater than all levels in $S$.
 
A new hook created on runner $i$ from level $k$ to level $k_G < k$ (runner $i$ has a gap at
$k_G$ in $\sigma_i(\lambda)$) has arm length:
\begin{equation}\label{eq:A1}
  A_1 = \vac{k_Ge+i}{ke+i}{\sigma_i(\lambda)} = \vac{k_Ge+i}{ke+i}{\lambda} + 1,
\end{equation}
since for any $k' \in S\cap(k_G,k)$ the empty space on runner $i$ is replaced by
one on runner $i-1$, so the only net change is the gap created at $ke+(i-1)$
(contributing $+1$).
 
A new hook on runner $i-1$ from level $k_B>k$ to level $k$ (runner $i-1$ has a bead
at $k_B$ in $\sigma_i(\lambda)$) has arm length:
\begin{equation}\label{eq:A2}
  A_2 = \vac{ke+(i-1)}{k_Be+(i-1)}{\sigma_i(\lambda)} = \vac{ke+i}{k_Be+(i-1)}{\lambda},
\end{equation}
since the only net change in the interval $(ke+(i-1), k_Be+(i-1))$ is the bead
created at $ke+i$ (contributing $-1$), which lies just outside the left endpoint of
the interval $(ke+i, k_Be+(i-1))$, and all other intermediate moved beads at
levels $k' \in S \cap (k, k_B)$ contribute a net of zero by the same argument.
We must show both $A_1$ and $A_2$ are divisible by $d$.
 
\medskip\noindent\textbf{Arm length $A_1$ (new hook on runner $i$).}
From \eqref{eq:A1}, it suffices to show $\vac{k_Ge+i}{ke+i}{\lambda} \equiv -1 \pmod{d}$.
Since runner $i$ is vacant at $k_G$ in $\lambda$, level $k_G$ is either neutral (runners $i-1$ and $i$ are both vacant) or contributes a  $+$ (runner $i-1$ occupied, runner $i$ vacant). Since we are assuming a new hook is created on runner $i$, a bead at level $k_G$ on runner $i-1$ (if it exists) does not move.
 
\medskip\noindent\textit{Sub-case 1a: runner $i-1$ vacant at $k_G$ (level $k_G$ is
neutral).}
The hook on runner $i-1$ from level $k$ to level $k_G$ exists in $\lambda$ (runner
$i-1$ is occupied at $k$ since $k \in S$, and is vacant at $k_G$), with arm length
$\vac{k_Ge+(i-1)}{ke+(i-1)}{\lambda} \equiv 0 \pmod{d}$.
Since $k_Ge+i$ is a gap and $ke+(i-1)$ is a bead in $\lambda$, Lemma~\ref{lem:shift} gives
\[
  \vac{k_Ge+i}{ke+i}{\lambda}
  = \vac{k_Ge+(i-1)}{ke+(i-1)}{\lambda} - 1
  \equiv -1 \pmod{d},
\]
so, by \eqref{eq:A1}, $A_1 \equiv 0 \pmod{d}$. 
 
 \medskip\noindent\textit{Sub-case 1b: runner $i-1$ is occupied at $k_G$
(level $k_G$ contributes a $+$).}
The new hook on runner $i$ from $k$ to $k_G$ exists in $\sigma_i(\lambda)$
because runner $i$ remains vacant at $k_G$ after the moves. Hence $k_G\notin S$.
Thus the $+$ at level $k_G$ is either matched, or it is an unmatched $+$ not
among the moved levels $S$. We treat these two possibilities separately.

First, suppose that the $+$ at level $k_G$ is matched. Then it is cancelled with
a $-$ at some level $\ell<k_G$, since the $i$-signature is read in increasing
level and cancellations are of the form $-+$. At level $\ell$, runner $i-1$ is
vacant and runner $i$ has a bead. Thus, on runner $i-1$ in $\lambda$, we have
hooks from level $k$ to level $\ell$ and from level $k_G$ to level $\ell$.
Their arm lengths are both divisible by $d$:
\[
  \vac{\ell e+(i-1)}{k e+(i-1)}{\lambda} \equiv 0 \pmod d, \quad
  \vac{\ell e+(i-1)}{k_G e+(i-1)}{\lambda} \equiv 0 \pmod d.
\]
Since $\ell<k_G<k$ and $k_Ge+(i-1)$ is a bead, subtracting the latter congruence from the former one, we obtain
\[
  \vac{k_Ge+(i-1)}{k e+(i-1)}{\lambda}
  \equiv 0 \pmod d.
\]
Since $k_Ge+i$ is a gap and $ke+(i-1)$ is a bead in $\lambda$,
Lemma~\ref{lem:shift} gives
\[
  \vac{k_Ge+i}{ke+i}{\lambda}
  =
  \vac{k_Ge+(i-1)}{ke+(i-1)}{\lambda}-1
  \equiv -1 \pmod d.
\]
Therefore, by \eqref{eq:A1}, $A_1 \equiv 0 \pmod{d}$.

Now suppose that the $+$ at level $k_G$ is unmatched. This case
can occur only if $b\geq 1$, since if $b=0$, then all unmatched $+$'s are moved. Since the $i$-reduced signature has the form $+^a-^b$, all
$b$ unmatched $-$'s are strictly greater than all levels in $S$. Let $m$ be one of these $b$ levels, so $m> k$ and runner
$i$ has a bead at level $m$. Since level $k_G$ contributes a $+$, runner $i$ is vacant at level
$k_G$ in $\lambda$; also runner $i$ is vacant at level $k$ because $k\in S$.
Therefore, on runner $i$ in $\lambda$, the bead at level $m$ gives hooks to
both levels $k$ and $k_G$. Their arm lengths are divisible by $d$:
\[
  \vac{ke+i}{me+i}{\lambda} \equiv 0 \pmod d,
  \quad
  \vac{k_Ge+i}{me+i}{\lambda} \equiv 0 \pmod d.
\]
Since $k_G<k<m$, subtracting the former congruence from the latter one, we obtain 
\[
  \vac{k_Ge+i}{ke+i+1}{\lambda} \equiv 0 \pmod d.
\]
Since $ke + i$ is a gap, we have
\[
  \vac{k_Ge+i}{ke+i}{\lambda}
  =
  \vac{k_Ge+i}{ke+i+1}{\lambda}
  -1
  \equiv -1 \pmod d.
\]
Hence, again by \eqref{eq:A1}, $A_1 \equiv 0 \pmod{d}$.

\medskip\noindent\textbf{Arm length $A_2$ (new hook on runner $i-1$).}
Recall that $A_2$ is the arm length of a hook on runner $i-1$ from level $k_B>k$ (which is a bead) to level $k$ (which is vacant after the moves). Note that $k_B \notin S$, since the beads at levels in $S$ on runner $i-1$ have all moved to
runner $i$, and so runner $i-1$ is vacant at those levels after the moves. By
\eqref{eq:A2} the arm length of such a hook in $\sigma_i(\lambda)$ equals
$\vac{ke+i}{k_Be+(i-1)}{\lambda}$, and we must show this is $\equiv 0 \pmod{d}$.
Since $k_B \notin S$ and runner $i-1$ has a bead at $k_B$, level $k_B$ either
contributes nothing to the $i$-signature (runner $i$ occupied) or is a
matched $+$ (runner $i$ vacant).
 
\medskip\noindent\textit{Sub-case 2a: runner $i$ occupied at $k_B$ (neutral).}
The hook on runner $i$ from $k_B$ to $k$ exists in $\lambda$ and has arm length
$\vac{ke+i}{k_Be+i}{\lambda} \equiv 0 \pmod{d}$.
Since $k_Be+(i-1)$ is a bead,
\[
A_2=\vac{ke+i}{k_Be+(i-1)}{\lambda}
= \vac{ke+i}{k_Be+i}{\lambda} \equiv 0 \pmod{d}.
\]
 
\medskip\noindent\textit{Sub-case 2b: runner $i$ vacant at $k_B$ (matched $+$).}
The $+$ at level $k_B$ is cancelled by a $-$ at some level $\ell$ with
$k < \ell < k_B$: runner $i$ is occupied and runner $i{-}1$ is vacant at $\ell$.
We know $\ell > k$ since $\ell$ was chosen to cancel with $k_B$, and were $k_B<k$
this could not happen since $k$ is unmatched.
Two hooks exist in $\lambda$: one on runner $i$ from level $\ell$ to level
$k \in S$ and the other on runner $i-1$ from level $k_B$ to level $\ell$. Their arm lengths are:
\[
\vac{ke+i}{\ell e+i}{\lambda} \equiv 0 \pmod d, \quad
\vac{\ell e+(i-1)}{k_Be+(i-1)}{\lambda} \equiv 0 \pmod d.
\]
Since $ke+i < \ell e+i < k_Be+(i{-}1)$
(the last inequality holds because $\ell<k_B$ and $e\ge 2$), adding these congruences gives
\[
A_2=\vac{ke+i}{k_Be+(i-1)}{\lambda} \equiv 0 \pmod{d}.
\]
 
\bigskip\bigskip
\noindent\textbf{Case 2 $a < b$: beads move left, from runner $i$ to runner $i-1$.}

A new hook on runner $i-1$ from level $k$ to level $k_G < k$ (runner $i-1$ has a gap
at $k_G$ in $\sigma_i(\lambda)$) has arm length:
\begin{equation}\label{eq:A1prime}
  A_1' = \vac{k_Ge+(i-1)}{ke+(i-1)}{\sigma_i(\lambda)} = \vac{k_Ge+(i-1)}{ke+(i-1)}{\lambda},
\end{equation}
since for any $k' \in S\cap(k_G,k)$ the empty space on runner $i-1$ is replaced by
one on runner $i$, so there is no net change, and the bead arriving at $ke+(i-1)$
lies at the right endpoint (excluded).
 
A new hook on runner $i$ from level $k_B > k$ to level $k$ (runner $i$ has a bead at
$k_B$ in $\sigma_i(\lambda)$) has arm length:
\begin{equation}\label{eq:A2prime}
  A_2' = \vac{ke+i}{k_Be+i}{\sigma_i(\lambda)} = \vac{ke+i}{k_Be+i}{\lambda},
\end{equation}
since the bead departing from $ke+i$ lies at the left endpoint (excluded), and
$ke+(i-1)$ (which gains a bead) lies strictly to the left of $ke+i$ and so
outside the interval, and all other intermediate moved beads contribute a net of zero
by the same argument.
We must show that both $A_1'$ and $A_2'$ are divisible by $d$.
 
\medskip\noindent\textbf{Arm length $A_1'$ (new hook on runner $i-1$).}
From \eqref{eq:A1prime}, it suffices to show
$\vac{k_Ge+(i-1)}{ke+(i-1)}{\lambda} \equiv 0 \pmod{d}$.
Since runner $i-1$ is vacant at $k_G$ after the move, we have $k_G \notin S$. Level $k_G$ therefore
either contributes nothing to the $i$-signature (runner $i$ also vacant) or
contributes a matched $-$ (runner $i$ occupied at $k_G$).
 
\medskip\noindent\textit{Sub-case 1$'$a: runner $i$ vacant at $k_G$ (level $k_G$ is neutral).}
The hook on runner $i$ from $k$ to $k_G$ exists in $\lambda$ and has arm length
$\vac{k_Ge+i}{ke+i}{\lambda} \equiv 0 \pmod{d}$.
Since $k_Ge+i$ is a gap and $ke+(i-1)$ is a gap, we get:
\[
  A_1'=\vac{k_Ge+(i-1)}{ke+(i-1)}{\lambda}
  = \vac{k_Ge+i}{ke+i}{\lambda}\equiv 0 \pmod{d}.
\]
 
\medskip\noindent\textit{Sub-case 1$'$b: runner $i$ occupied at $k_G$ (level $k_G$ is a matched $-$).}
The matched $-$ at $k_G$ is paired with a $+$ at some $\ell \in (k_G, k)$ (runner
$i-1$ is a bead, and runner $i$ is vacant at $\ell$). Two hooks exist in $\lambda$: one on runner $i-1$ from $\ell$ to $k_G$ and the other on runner $i$ from $k$ to $\ell$; thus
\[
\vac{k_Ge+(i-1)}{\ell e+(i-1)}{\lambda} \equiv 0 \pmod d, \quad
\vac{\ell e+i}{ke+i}{\lambda} \equiv 0 \pmod d.
\]
Since $\ell e+(i-1)$ is a bead and $\ell e+i$ is a gap in $\lambda$, adding these congruences, we obtain
\[
  \vac{k_Ge+(i-1)}{ke+i}{\lambda} \equiv 1 \pmod d.
\]
As $ke+(i-1)$ is a gap in $\lambda$, we immediately deduce that
\[
A_1' = \vac{k_Ge+(i-1)}{ke+(i-1)}{\lambda} = \vac{k_Ge+(i-1)}{ke+i}{\lambda} - 1 \equiv 0 \pmod d.
\]

\medskip\noindent\textbf{Arm length $A_2'$ (new hook on runner $i$).}
Recall that $A_2'$ is the arm length of a hook on runner $i$ from level $k_B>k$ (which is a bead) to level $k$ (which is vacant after the moves). Note that $k_B \notin S$, since the beads at levels in $S$ have all moved to runner
$i-1$ and so runner $i$ is vacant at those levels after the moves. By \eqref{eq:A2prime}
the arm length of such a hook in $\sigma_i(\lambda)$ equals
$\vac{ke+i}{k_Be+i}{\lambda}$, and we must show this is $\equiv 0 \pmod{d}$.
Since $k_B \notin S$ and runner $i$ has a bead at $k_B$, level $k_B$ either
contributes nothing to the $i$-signature (runner $i-1$ occupiedl) or
contributes a $-$ (runner $i-1$ vacant).
 
\medskip\noindent\textit{Sub-case 2$'$a: runner $i-1$ occupied at $k_B$ (neutral).}
The arm length of the hook of $\lambda$ on runner $i-1$ from $k_B$ to $k$ (runner $i-1$ has a gap at $k$ since $k \in S$) is $\vac{ke+(i-1)}{k_Be+(i-1)}{\lambda}\equiv 0 \pmod d$.
Since $ke+i$ and $k_Be+(i-1)$ are beads of $\lambda$, we obtain
\[
  A_2'=\vac{ke+i}{k_Be+i}{\lambda}
  = \vac{ke+(i-1)}{k_Be+(i-1)}{\lambda}.
\]
 
\medskip\noindent\textit{Sub-case 2$'$b: runner $i-1$ vacant at $k_B$ (level $k_B$
contributes $-$).}
We further split according to whether the $-$ at level $k_B$ is matched or unmatched.

First suppose that the $-$ at level $k_B$ is matched. Then it is cancelled with a $+$ at some $\ell > k_B$ (runner
$i-1$ has a bead, and runner $i$ is vacant at $\ell$). On runner $i-1$ in $\lambda$ we have hooks that go from level $\ell$ to level $k$ (since $k\in S$) and from level $\ell$ to level $k_B$. Their respective arm lengths are:
\[
  \vac{ke+(i-1)}{\ell e+(i-1)}{\lambda} \equiv 0 \pmod{d}, \quad
  \vac{k_Be+(i-1)}{\ell e+(i-1)}{\lambda} \equiv 0 \pmod{d}.
\]
Subtracting the latter from the former yields:
\[
\vac{ke+(i-1)}{k_Be+i}{\lambda} \equiv 0 \pmod{d}.
\]
Since $ke+i$ is a bead in $\lambda$,
\[
  A_2'=\vac{ke+i}{k_Be+i}{\lambda} =  \vac{ke+(i-1)}{k_Be+i}{\lambda} \equiv 0 \pmod{d}.
\]

Now suppose that the $-$ at level $k_B$ is unmatched. This case can only occur if $a\geq 1$, since if $a=0$, then all unmatched $-$'s are moved. Since the $i$-reduced signature has the form $+^a -^b$, all $a$ unmatched $+$'s are at levels strictly smaller than all levels in $S$. Let $m$ be one of these $a$ levels, so $m<k$ and runner $i$ is vacant at $m$. Two hooks with the following corresponding arm lengths exist on runner $i$:
\[
  \vac{me+i}{ke+i}{\lambda} \equiv 0 \pmod{d}, \quad
  \vac{me+i}{k_Be+i}{\lambda} \equiv 0 \pmod{d}.
\]
Since $ke+i$ is a bead, subtracting the former from the latter yields:
\[
  A_2'=\vac{ke+i}{k_Be+i}{\lambda} \equiv 0 \pmod{d}.
\]

\end{proof}

We immediately deduce the $e$-core independence of the number of $d$-balanced $e$-regular partitions in a block of $e$-weight $w\geq 0$. We determine this number in the next section.

\begin{corollary}
\label{cor:samenumbereveryblockweightw}
For integers $d,e>1$ and $w \geq 0$, the number of $d$-balanced $e$-regular partitions is the same in every block of $e$-weight $w$; in particular, it depends only on $d$, $e$ and $w$.
\end{corollary}
 
\begin{proof}
This is a combination of Theorem~\ref{thm:main} and Proposition~\ref{prop:crystal-bijection}.
\end{proof}

\section{RoCK Blocks}\label{se:RoCK}
Having shown that the number of $d$-balanced $e$-regular partitions is the same for every block
of $e$-weight $w$, we strategically choose to count them in blocks known as RoCK blocks
(also called Rouquier blocks), where the abacus structure is particularly well-suited to determining arm-lengths. We prove:
 
\begin{theorem}
  \label{thm:maincountRouquier}
  Let $d,e>1$ and $w\geq 0$ be integers. The number of $d$-balanced $e$-regular partitions in a block of $e$-weight $w$ is
  \[
    \binom{w + \left\lfloor\frac{e-1}{d}\right\rfloor - 1}{w}.
  \]
\end{theorem}
 
To prove this theorem we focus on the RoCK block of $e$-weight $w$. We represent its
$e$-core on an abacus where runner $i \in \{0, 1, \ldots, e-1\}$ has $w(i+1)$ beads,
all slid up to the top. For example, Figure~\ref{fig:rock-core} shows the case $e=5$,
$w=3$.
 
\begin{figure}[ht]
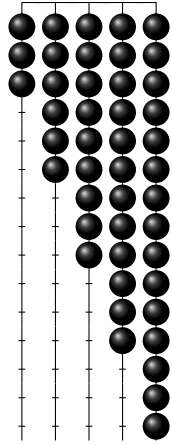

\centering
\abacus(lmmmr,%
bbbbb,bbbbb,bbbbb,%
nbbbb,nbbbb,nbbbb,%
nnbbb,nnbbb,nnbbb,%
nnnbb,nnnbb,nnnbb,%
nnnnb,nnnnb,nnnnb)
\caption{The RoCK block $5$-core for $w=3$, with runners $0,1,2,3,4$ having
$3,6,9,12,15$ beads respectively.}
\label{fig:rock-core}
\end{figure}

This abacus representation has a multiple of $e$ beads (and thus corresponds to the canonical $\beta$-set) whenever $e$ is odd or $w$ is even. In the remaining cases, one can use the $\beta$-set with respect to $r=e/2$. Consequently, when referring to an abacus display of partition $\lambda$, we use the $\beta$-set $B_{e/2}(\lambda)$ when $e$ is even and $w$ is odd, and the canonical $\beta$-set $B_0(\lambda)$, otherwise.
We analogously update the definition of $\vac{u}{v}{\lambda}$ from the previous section. We can still use Lemma~\ref{lem:hook}, as explained in Remark~\ref{re:Other r}.
 
What makes this block particularly tractable is that any partition in
this block has the property that every bead on runner $j$ at level $k$ has beads on
every runner to its right at every level $k' \leq k$. Put informally, no bead
can move far enough down the abacus to pass the wall of beads on its right. This makes determining arm lengths much easier:
 
\begin{lemma}
\label{lemma:armlengthinrockblocks}
Suppose that $\lambda$ is a partition in a RoCK block with an $ae$-hook of $\lambda$ on the runner $i\in \{ 0,1,\dots, e-1\}$ going from level $k$ to level $k-a$. Suppose there are $t$ vacant
positions of $\lambda$ on runner $i$ strictly between levels $k-a$ and $k$. Then the arm length
of this hook is $ai+t$.
\end{lemma}
 
\begin{proof}
The proof is illustrated in Figure \ref{fig:rock-arm-length}. By Lemma~\ref{lem:hook}, the arm length equals $\vac{e(k-a)+i}{ek+i}{\lambda}$, the
number of vacant positions of $\lambda$ strictly between $e(k-a)+i$ and $ek+i$.

Using the property of RoCK blocks mentioned before this lemma and the definition of $t$ we conclude that:
\begin{itemize}
    \item All positions of $\lambda$ on runner $i+1, i+2,\ldots,e-1$ at levels $k-a,k-a+1,\ldots,k-1$ are \emph{occupied} (since position of $\lambda$ on runner $i$ on level $k$ is occupied).
    \item Precisely $t$ vacant positions of $\lambda$ on runner $i$ lie strictly between levels $k-a$ and $k$.
    \item All positions of $\lambda$ on runner $0,\ldots,i-1$ at levels $k-a+1,k-a+2,\ldots,k$ are \emph{vacant} (since position of $\lambda$ on runner $i$ on level $k-a$ is vacant).
\end{itemize}
Hence the arm length is $ai + t$.
\end{proof}
 
\begin{figure}[ht]
\centering
\begin{tikzpicture}[>=Stealth, scale=1.0]
 
  \def\rs{0.72}
  \def\rw{0.85}
  \def\br{0.13}
  \def\dl{0.10}
  \def\hw{0.28}
 
  \def\yka{0}
  \def\ykthree{-\rs}
  \def\yktwo{-2*\rs}
  \def\ykone{-3*\rs}
  \def\yk{-4*\rs}

  \fill[red!15, rounded corners=2pt]
    (-\hw, \ykthree+0.24) rectangle (2*\rw+\hw, \yk-0.24);

  \fill[blue!15, rounded corners=2pt]
    (3*\rw-\hw, \ykthree+0.24) rectangle (3*\rw+\hw, \ykthree-0.24);
  \fill[blue!15, rounded corners=2pt]
    (3*\rw-\hw, \ykone+0.24) rectangle (3*\rw+\hw, \ykone-0.24);

  \foreach \x in {0, \rw, 2*\rw, 3*\rw, 4*\rw} {
    \draw[thick] (\x, 0.35) -- (\x, -4.35*\rs);
  }
 
  \node[above, font=\small] at (0,     0.35) {$0$};
  \node[above, font=\small] at (\rw,   0.35) {$1$};
  \node[above, font=\small] at (2*\rw, 0.35) {$\cdots$};
  \node[above, font=\small] at (3*\rw, 0.35) {$i$};
  \node[above, font=\small] at (4*\rw, 0.35) {$i{+}1$};
  \node[above, font=\small] at (5*\rw, 0.35) {$\cdots$};
 
  \node[left, font=\small] at (-0.15, \yka)     {$k{-}4$};
  \node[left, font=\small] at (-0.15, \ykthree) {$k{-}3$};
  \node[left, font=\small] at (-0.15, \yktwo)   {$k{-}2$};
  \node[left, font=\small] at (-0.15, \ykone)   {$k{-}1$};
  \node[left, font=\small] at (-0.15, \yk)      {$k$};
 
  \draw (0-\dl,     \yka) -- (0+\dl,     \yka);
  \draw (\rw-\dl,   \yka) -- (\rw+\dl,   \yka);
  \draw (2*\rw-\dl, \yka) -- (2*\rw+\dl, \yka);
  \draw (3*\rw-\dl, \yka) -- (3*\rw+\dl, \yka);
  \fill (4*\rw, \yka) circle (\br);
 
  \draw (0-\dl,     \ykthree) -- (0+\dl,     \ykthree);
  \draw (\rw-\dl,   \ykthree) -- (\rw+\dl,   \ykthree);
  \draw (2*\rw-\dl, \ykthree) -- (2*\rw+\dl, \ykthree);
  \draw (3*\rw-\dl, \ykthree) -- (3*\rw+\dl, \ykthree);
  \fill (4*\rw, \ykthree) circle (\br);
 
  \draw (0-\dl,     \yktwo) -- (0+\dl,     \yktwo);
  \draw (\rw-\dl,   \yktwo) -- (\rw+\dl,   \yktwo);
  \draw (2*\rw-\dl, \yktwo) -- (2*\rw+\dl, \yktwo);
  \fill (3*\rw, \yktwo) circle (\br);
  \fill (4*\rw, \yktwo) circle (\br);
 
  \draw (0-\dl,     \ykone) -- (0+\dl,     \ykone);
  \draw (\rw-\dl,   \ykone) -- (\rw+\dl,   \ykone);
  \draw (2*\rw-\dl, \ykone) -- (2*\rw+\dl, \ykone);
  \draw (3*\rw-\dl, \ykone) -- (3*\rw+\dl, \ykone);
  \fill (4*\rw, \ykone) circle (\br);
 
  \draw (0-\dl,     \yk) -- (0+\dl,     \yk);
  \draw (\rw-\dl,   \yk) -- (\rw+\dl,   \yk);
  \draw (2*\rw-\dl, \yk) -- (2*\rw+\dl, \yk);
  \fill[white] (3*\rw, \yk) circle (\br);
  \draw        (3*\rw, \yk) circle (\br);
  \fill (4*\rw, \yk) circle (\br);

  \draw[->, thick, blue!75!black]
    (3*\rw, \yk+\br+0.04) -- (3*\rw, \yka-\br-0.04);

  \def\labelx{5.2}
  \def\tickend{4.85}

  \draw[->] (\tickend, \yka-0.22) -- (3*\rw+\br+0.05, \yka-0.22)
    -- (3*\rw+\br+0.05, \yka);
  \node[right, font=\scriptsize] at (\labelx, \yka-0.22) {to here};

  \draw[->] (\tickend, \yk+0.22) -- (3*\rw+\br+0.05, \yk+0.22)
    -- (3*\rw+\br+0.05, \yk);
  \node[right, font=\scriptsize] at (\labelx, \yk+0.22) {hook goes from here};

  \node[font=\scriptsize, red!70!black, left] at (-0.35, 0.5*\ykthree+0.5*\yk)
    {$ai$ vacancies};

  \node[font=\scriptsize, blue!70!black] at (\labelx+0.8, 0.5*\ykthree+0.5*\ykone)
    {$t=2$ gaps};

  \foreach \x in {0, \rw, 2*\rw, 3*\rw, 4*\rw} {
    \node at (\x, -4.65*\rs) {$\vdots$};
  }
 
\end{tikzpicture}
\caption{Illustration of Lemma~\ref{lemma:armlengthinrockblocks} with $a=4$ and
$t=2$. The arm length is $ai+t$.}
\label{fig:rock-arm-length}
\end{figure}
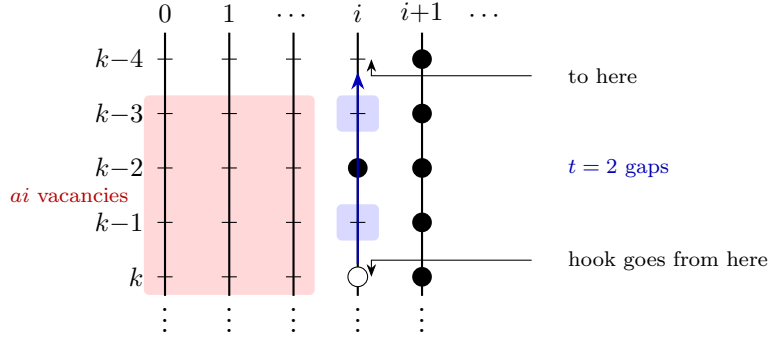
We now determine exactly which abacus configurations are allowed for $d$-balanced $e$-regular partitions in a RoCK block. If runner $0$ of partition $\lambda$ in a RoCK block has a bead on level $k$, then all $e$ runners on level $k$ are occupied by a bead; in turn, if there is a gap on runner $0$ above a bead, then $\lambda$ is not $e$-regular (see Remark~\ref{re:beta-set}(4)). So runner $0$ of an $e$-regular partition in a RoCK block must remain \textit{flush to the top} (no bead on runner $0$ in the $e$-core can slide down).
The following two lemmas progressively restrict the remaining runners.
 
\begin{lemma}\label{lemma:whichrunners}
Suppose $\lambda$ is a $d$-balanced partition in a RoCK block and $i \in \{0, 1, \dots, e-1\}$. If $d \nmid i$ then the beads on runner $i$ are all flush to the top of the abacus. Assuming further that $\lambda$ is also $e$-regular, the only runners, on which beads may slide down are $d, 2d, \ldots, \left\lfloor\frac{e-1}{d}\right\rfloor d$.
\end{lemma}
 
\begin{proof}
If any bead on runner $i$ can slide up on the abacus, that is, if there is a vacant position above an occupied position, we obtain a corresponding rim
$e$-hook. By Lemma~\ref{lemma:armlengthinrockblocks} applied with $a=1$ and $t=0$, this $e$-hook
has arm length $i$. For $\lambda$ to be $d$-balanced we need $d \mid i$.
Together with the discussion about $e$-regular partitions, this finishes the proof.
\end{proof}

\begin{lemma}\label{le:notwo}
    Suppose $\lambda$ is a $d$-balanced partition in a RoCK block. There is at most one gap on the same runner and above any given bead of $\lambda$.
\end{lemma}

\begin{proof}
    Suppose for contradiction, that this fails to hold. Then there is a runner $i$ with a bead on level $k$ and gaps on levels $k-1$ and $k-b$ (with $b>1$). Making $b$ smaller if necessary, we assume that apart from levels $k-1$ and $k-b$ all levels $k, k-1, \ldots, k-b$ of runner $i$ contain beads. By Lemma~\ref{lemma:armlengthinrockblocks} applied with $a=b$ and $t=1$, there is an $e$-divisible hook with arm length $bi+1$. Thus $d\mid bi+1$, but also $d\mid i$ by Lemma~\ref{lemma:whichrunners}(1), contradicting $d>1$. This finishes the proof.
\end{proof}
 
Lemma~\ref{le:notwo} states that on any runner of a $d$-balanced partition in a RoCK block, the only permitted pattern is:
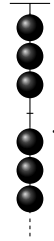
\begin{center}
  \abacus(m,b,b,b,n,b,b,b,v).
\end{center}
That is, at most a single gap in the runner, with beads below it. Such configurations
correspond to some number of beads $m \geq 0$, each having slid down once
from their $e$-core position, contributing $m$ to the $e$-weight. We are ready to prove our first main theorem.
 
\medskip\noindent\textbf{Proof of Theorem~\ref{thm:maincountRouquier}.}
By Corollary~\ref{cor:samenumbereveryblockweightw}, it suffices to count the desired partitions in the RoCK block of $e$-weight $w$. By Lemma~\ref{lemma:whichrunners}, only runners $d, 2d, \ldots, \left\lfloor\frac{e-1}{d}\right\rfloor d$
may have any beads in the $e$-core slid down. By Lemma~\ref{le:notwo},
each such runner has at most a single vacant position with beads below it, so on that runner some number of beads in the $e$-core have each slid down exactly once.

Conversely, any partition in this RoCK block obtained by sliding some number of beads in the $e$-core on runners $d, 2d, \ldots, \left\lfloor\frac{e-1}{d}\right\rfloor d$ down exactly once is $d$-balanced: any $e$-divisible hook correspond to a bead at level $k$ and on runner $i$ such that $d\mid i$ and the \emph{unique} (if it exists) gap above it, say, on level $k-a$; by Lemma~\ref{lemma:armlengthinrockblocks} applied with $a$ and $t=0$, the arm length is $ai$, which is divisible by $d$. These partitions are also $e$-regular as runner $0$ remains flash to the top.

Thus the number of $d$-balanced $e$-regular partitions in the RoCK block of $e$-weight $w$ is the number of ways to write $w$ as a sum of $\left\lfloor\frac{e-1}{d}\right\rfloor$ non-negative integers
(which are the numbers of beads that slid down on the allowed runners). By the standard stars-and-bars argument, this equals
\[
  \binom{w + \left\lfloor\frac{e-1}{d}\right\rfloor - 1}{w},
\]
as required. \qed
\begin{remark}
\label{remark:equotient1d}
In the language of $e$-quotients, we have shown that a $d$-balanced $e$-regular partition in a RoCK block has nonempty $e$-quotient
component only at non-zero runners which are a multiple of $d$, and each such component is
of the form $(1^m)$ for some $m \geq 0$.
\end{remark}

There is a neat way to find the number of all (both $e$-regular and $e$-singular) $d$-balanced partitions in a block by using a map $\psi$, which appends $e$ ones to a partition. That is, if $\lambda = (\lambda_1, \lambda_2, \dots, \lambda_t)$, then $\psi(\lambda) = (\lambda_1, \lambda_2, \dots, \lambda_t, 1,\dots, 1)$, where there are $e$ extra ones. We make a simple observation

\begin{lemma}\label{le:hook lengths}
    Let $\lambda$ be a partition and $(r,c)\in [\lambda]$ be its node. The arm length $a_{rc}(\lambda)$ equals the arm length $a_{rc}(\psi(\lambda))$. Moreover, $e$ divides the hook length $h_{rc}(\lambda)$ if and only if it divides the hook length $h_{rc}(\psi(\lambda))$.
\end{lemma}

\begin{proof}
    The first assertion is clear. For the second one, we use that $h_{rc}(\psi(\lambda)) = h_{rc}(\lambda)$ if $r>1$ and $h_{rc}(\psi(\lambda)) = h_{rc}(\lambda)+e$ if $r=1$.
\end{proof}

With this observation, we can describe the behaviour of $\psi$ applied to $d$-balanced partitions. In the proof, we work with canonical $e$-abaci.

\begin{lemma}\label{le:add a column}
    Let $\gamma$ be an $e$-core partition and $w\geq 0$ an integer. The map $\psi$ restricts to a bijection from the set of $d$-balanced partitions with $e$-core $\gamma$ and $e$-weight $w$ to the set of $d$-balanced $e$-singular partitions with $e$-core $\gamma$ and $e$-weight $w+1$.
\end{lemma}

\begin{proof}
    By Remark~\ref{re:rim hook removal}, $\phi$ preserves $e$-cores and increases $e$-weight by $1$ and always inputs an $e$-singular partition. Moreover, for any partition $\lambda$ we have that $\lambda$ is $d$-balanced if and only if $\psi(\lambda)$ is $d$-balanced: indeed, this is Lemma~\ref{le:hook lengths} and the fact that the only $e$-divisible hook of $\psi(\lambda)$ at node $(r,c)\in[\psi(\lambda)]\setminus [\lambda]$ has arm length $0$.

    Since this restriction of $\psi$ is clearly injective, we only need to show that any $d$-balanced $e$-singular partition $\mu$ lies in the image of $\psi$, that is, $\mu$ has at least $e$ parts equal to $1$. To do so, we pass to the $e$-runner abacus representation of $\mu$.

    Since $\mu$ is $e$-singular, there is a gap in some position $f$ on the abacus of $\mu$ such that all positions $f+1, f+2,\dots, f+e$ are occupied with beads. Take $f$ to be the least such position. If all positions less than $f$ are occupied, we are done (see Remark~\ref{re:beta-set}(4)); we show this must be the case. Otherwise, there is the least $j>0$ such that position $f-j$ is a gap. By the minimal choices of $f$ and $j$ we have $j\leq e$. If $j<e$, then $\vac{f-j}{f-j+e}{\mu}=1$ is an arm length of an $e$-hook, contradicting that $\mu$ is $d$-balanced. If $j=e$, then $\vac{f-e}{f+e}{\mu}=1$ is an arm length of a $2e$-hook, again, contradiction that $\mu$ is $d$-balanced. This finishes the proof.
\end{proof}

Combining Lemma~\ref{le:add a column} and Theorem~\ref{thm:maincountRouquier} we immediately get the desired number.

\begin{corollary}\label{co:all d-balanced}
Let $d, e >1$ and $w\geq 0$ be integers. The number of $d$-balanced partitions in a block of $e$-weight $w$ is
  \[
    \binom{w + \left\lfloor\frac{e-1}{d}\right\rfloor}{w}.
  \]
\end{corollary}

\begin{proof}
    We induct on $w\geq 0$. If $w=0$, the only ($e$-core) partition in the block is $d$-balanced. Assuming that $w\geq 1$ and the result holds for $w-1$, the number of $d$-balanced $e$-regular partitions in a given block of $e$-weight $w$ is $\binom{w + \left\lfloor\frac{e-1}{d}\right\rfloor-1}{w}$ by Theorem~\ref{thm:maincountRouquier}. The number of $d$-balanced $e$-singular partitions in the same block is by Lemma~\ref{le:add a column} and the induction hypothesis $\binom{w + \left\lfloor\frac{e-1}{d}\right\rfloor-1}{w-1}$. Adding these two numbers proves the result.  
\end{proof}

Having fully enumerated $d$-balanced and $d$-balanced $e$-regular partitions, one may ask what happens if we want to consider partitions such that all their $e$-divisible hooks have arm lengths congruent to some fixed residue $r$ modulo $d$? If $r\equiv 0 \pmod{d}$, these are precisely $d$-balanced partitions. If $r\equiv -1 \pmod{d}$, these are the $d$-shift balanced partitions from \cite{turek2025mullineuxmapdbalancedpartitions}. As mentioned in the introduction, $d$-shift balanced $e$-restricted partitions are, in a way, dual to $d$-balanced $e$-regular partitions, namely, the composition of the Mullineux map and conjugation of partitions gives a bijection between these two families of partitions. For $d=2$ this is a consequence of the main results of \cite{turek2025mullineuxmapdbalancedpartitions}, while for $d>2$ one has to, in addition, use some upcoming results from the ongoing work of Gustavsson, Law, Putignano, Speyer and the second author.

Thanks to this duality, the number of $d$-shift balanced $e$-restricted partitions in a given block is the same as the number of $d$-balanced $e$-regular partitions in the same block. In particular, it depends only on the $e$-weight and not the $e$-core and is given by Theorem~\ref{thm:maincountRouquier}. One can further compute that there are
\[
\binom{w+\left\lfloor\frac{e}{d}\right\rfloor-1}{w}
\]
$d$-shift balanced partitions in a block of $e$-weight $w$. These results can be derived directly: one shows that the crystal (affine) reflections preserve \emph{conjugates} of $d$-shift balanced partitions using an analogous case-by-case analysis as in the proof of Theorem~\ref{thm:main} and then pivots to RoCK blocks to count $d$-shift balanced $e$-restricted partitions there. The `conjugate' version of the map $\psi$ can then be used to move from $e$-restricted to all such partitions. This approach also provides direct proof that the composition of the Mullineux map and conjugation of partitions gives a bijection between $d$-balanced $e$-regular and $d$-shift balanced $e$-restricted partitions thanks to the identity $\sigma_i\circ m_e = m_e\circ \sigma_{-i}$ (\cite[Theorem~4.7]{KleshchevBranchingIII96}), where $m_e$ is the Mullineux map and the index $-i$ is taken modulo $e$.

For residue classes other than $0$ and $-1$ modulo $d$, it is not true that the number of partitions with all arm lengths of its $e$-divisible hooks belonging to this class in a given block is independent of the $e$-core. For example, let $e=7$, $d=3$, and $w=2$, and require every $7$-divisible hook to have arm length congruent to $1$ modulo $3$. In the block with $7$-core $(1)$ and $e$-weight $2$ there are four such partitions:
\[
(3,2^5,1^2),\quad (6,2^2,1^5),\quad (6,5,3,1),\quad (8,2,1^5).
\]
In the block with $7$-core $(2)$ and $e$-weight $2$ there are only three:
\[
(3^2,2^4,1^2),\quad (6,3,2,1^5),\quad (6,5,4,1).
\]
The affine reflection $\sigma_1$ maps $(8,2,1^5)$ to $(9,2,1^5)$, and the latter has a $7$-hook in the first row with arm length $6$. Interestingly, the other six listed partitions pair up under this reflection. 
So while for any residue class modulo $d$, it is straightforward to compute the number of such partitions in a RoCK block, in a similar way as done for residue class $0$, this does not seem to yield the number for any arbitrary block when the residue class is not $0$ or $-1$ modulo $d$.

\section{Combinatorics of odd sequences}\label{se:odd}
 
There is a remarkable $e$-refinement of the number of odd parts in a partition, these are the \textit{odd sequences}. They were introduced and linked to the Mullineux map in \cite{turek2025mullineuxmapdbalancedpartitions} and in \cite{hemmerturek2026} were found to be crucial in constructing columns for decomposition matrices of symmetric groups and, conjecturally, give complete descriptions for indecomposable summands of so-called twisted Foulkes modules $H^{(2^m;k)}$ (proved in case $k=0$).
 
Odd sequences refine the statistic that counts how many odd parts a partition has by looking at the $e$-residues at the end of the odd rows in the Young diagram. The connection to $d$-balanced partitions for $d=2$ will emerge through Proposition~\ref{prop:turkdbalancedmaximaloddsequence}.
 
Let $\lambda \vdash n$ and $e > 1$ be odd.
 
\begin{definition}
\label{def:oddsequence}
The \emph{odd sequence} of $\lambda$ is a composition $O(\lambda):=(n_0, n_1, n_2, \cdots, n_{e-1})$ where $n_i$ is the number of parts $\lambda_j$ with $\lambda_j$ odd and the $e$-residue of the right-most node in row $j$ equal to $i$, i.e. $\lambda_j - j \equiv i \pmod{e}$.
\end{definition}
 
\begin{remark}
\label{remark: oddsequencecountsoddbeads}
If we represent $\lambda$ on an $e$-abacus with a multiple of $e$ beads, then $n_i$ counts the number of \textit{odd beads} on runner $i$, where an odd bead is defined as a bead with an odd number of vacant spaces in lower-numbered positions on the abacus.
\end{remark}
 
Now fix an $e$-core $\gamma$ and a composition $\theta=(n_0, n_1, \cdots, n_{e-1})$. It is an easy exercise (proved in more generality in \cite[Corollary~2.11]{turek2025mullineuxmapdbalancedpartitions}) that we can, starting with the abacus display for $\gamma$, slide beads down in such a way to ensure there are $n_i$ odd beads on runner $i$. That is:
 
\begin{lemma}
\label{lem:omegaOiswelldefined}
  There exists a partition $\lambda$ with $e$-core $\gamma$ and odd sequence $O(\lambda)=\theta$.
\end{lemma}
 
Lemma \ref{lem:omegaOiswelldefined} is false for $e$ even, as adding a rim $e$-hook preserves the parity of the sum of entries of $O(\lambda)$, so starting with the odd sequence of $\gamma$, we can only obtain odd sequences with equal parity of the sum of their entries. (This is in fact the only constraint needed to make Lemma~\ref{lem:omegaOiswelldefined} work for even $e$; see \cite[Corollary~2.11]{turek2025mullineuxmapdbalancedpartitions}.) For odd $e$, by contrast, adding a rim $e$-hook changes this parity.
Lemma \ref{lem:omegaOiswelldefined} justifies the following definitions:
 
\begin{definition}
\label{def:omega_O}
   Let $w_\theta(\gamma)$ denote the minimum $e$-weight of a partition with $e$-core $\gamma$ and odd sequence $\theta$.
\end{definition}
 
\begin{definition} 
\label{def:Ethetagamma}
   Let $\mathcal{E}_\theta(\gamma)$ denote the set of all partitions with $e$-core $\gamma$, $e$-weight $w_\theta(\gamma)$, and odd sequence $\theta$. Observe that each of these is a partition of $|\gamma|+ew_\theta(\gamma).$
\end{definition}
 
There is another way to look at Definition~\ref{def:omega_O}. We can consider the set of partitions with $e$-core $\gamma$ and $e$-weight $w$ as $w=0,1,2,\ldots$. For each $w$, we can look at the odd sequences that occur among these partitions. A sequence $\theta$ will appear for the first time at $e$-weight $w_\theta(\gamma)$, where we call it \textit{new}. We will count how many total odd sequences there are for each $w$, and how many are new. We do the same for odd sequences that contain a zero.
All four numbers will be independent of $\gamma$.
 
\begin{definition}
\label{def:Se(gamma)Tegamma}
For $w\ge 0$, let $S_e^\gamma(w)$ be the set of odd sequences arising from partitions with $e$-core $\gamma$ and $e$-weight $w$. Define
\[
T_e^\gamma(w):=\lvert S_e^\gamma(w)\rvert
\]
and
\[
Z_e^\gamma(w):=\lvert\left\{(n_0,n_1,\ldots,n_{e-1}) \in S_e^\gamma(w):\text{ there exists } i \text{ such that } n_i=0\right\}\rvert.
\]
\end{definition}

We define similar quantities for new odd sequences.

\begin{definition}
\label{def:newzerosequences}
Let $M_e^{\gamma}(w)$ denote the number of new odd sequences in the block of $e$-core $\gamma$ and $e$-weight $w$, i.e. the number of odd sequences $\theta \in S_e^\gamma(w)$ such that $\theta \notin S_e^\gamma(w')$ for all $w'<w$, and let $N_e^{\gamma}(w)$ be the number of these sequences which contains a zero.
\end{definition}

\begin{example}\label{ex:oddsequence-example-e5}
Take \(e=5\) and \(5\)-core \(\gamma=(4,2,1)\). Using the GAP algebra system \cite{GAP4} to verify the computations for
$e$-weights \(w=0,1,2,3\), we obtain the following sets of odd sequences.
 
For \(w=0\) (so \(n=7\)),
\[
S_5^\gamma(0)=\{(0,0,0,1,0)\}.
\]
Thus
\[
T_5^\gamma(0)=1,\qquad M_5^\gamma(0)=1,\qquad N_5^\gamma(0)=1,
\]
and the unique new sequence (containing a zero) is
\[
(0,0,0,1,0).
\]
For \(w=1\) (so \(n=12\)),
\[
S_5^\gamma(1)=\{(0,0,0,2,0),\ (1,1,0,0,0),\ (1,1,1,2,1)\},
\]
so
\[
T_5^\gamma(1)=3,\qquad M_5^\gamma(1)=3,\qquad N_5^\gamma(1)=2,
\]
with new sequences containing a zero
\[
(0,0,0,2,0),\qquad (1,1,0,0,0).
\]
For \(w=2\) (so \(n=17\)),
\[
\begin{aligned}
S_5^\gamma(2)=\{&(0,0,0,1,0),\ (1,0,0,2,0),\ (1,1,0,1,0),\ (1,1,1,3,1),\\
&(2,1,0,1,1),\ (2,2,1,1,1),\ (2,2,2,3,2)\},
\end{aligned}
\]
hence
\[
T_5^\gamma(2)=7,\qquad M_5^\gamma(2)=6,\qquad N_5^\gamma(2)=3,
\]
and the new sequences containing a zero are
\[
(1,0,0,2,0),\qquad (1,1,0,1,0),\qquad (2,1,0,1,1).
\]
For \(w=3\) (so \(n=22\)),
\[
\begin{aligned}
S_5^\gamma(3)=\{&(0,0,0,2,0),\ (1,0,0,1,0),\ (1,0,0,2,1),\ (1,1,0,0,0),\ (1,1,1,2,1),\\
&(2,1,0,1,0),\ (2,1,0,2,1),\ (2,1,1,3,1),\ (2,2,1,2,1),\ (2,2,2,4,2),\\
&(3,2,1,2,2),\ (3,3,2,2,2),\ (3,3,3,4,3)\},
\end{aligned}
\]
so
\[
T_5^\gamma(3)=13,\qquad M_5^\gamma(3)=10,\qquad N_5^\gamma(3)=4,
\]
with new sequences containing a zero
\[
(1,0,0,1,0),\qquad (1,0,0,2,1),\qquad (2,1,0,1,0),\qquad (2,1,0,2,1).
\]
These values follow patterns $N_5^\gamma(w)=w+1$, $M_5^\gamma(w) = \binom{w+1}{2}$ and $T_5^\gamma(w) = T_5^\gamma(w-2) + M_5^\gamma(w)$. In fact, the last equality can be strengthened as: $S_5^\gamma(w)$ is the union of the set of its new odd sequences and $S_5^\gamma(w-2)$. We also observe that the new odd sequences in $S_5^{\gamma}(w)$ without any zeros are the new odd sequences in $S_5^{\gamma}(w-1)$ with all entries increased by one.
\end{example}

We proceed to prove the observed properties of our three functions in Example~\ref{ex:oddsequence-example-e5} in a general setting. The formulas for $M_e^{\gamma}(w)$ and $N_e^\gamma(w)$ follow from the following surprising characterization of $2$-balanced partitions. This characterization implicitly uses the fact that the sets $\mathcal{E}_\theta(\gamma)$ have a unique maximal element in the dominance order; see \cite[Proposition~6.14]{turek2025mullineuxmapdbalancedpartitions}. We omit the definition of the dominance order here; it suffices to treat it as some partial order on the set of partitions.

\begin{proposition}[\cite{turek2025mullineuxmapdbalancedpartitions}, Proposition~1.6]
\label{prop:turkdbalancedmaximaloddsequence}
Suppose $\lambda$ has odd sequence $\theta$ and $e$-core $\gamma$. Then $\lambda$ is a maximal element in the dominance order of $\mathcal{E}_\theta(\gamma)$ if and only if $\lambda$ is $2$-balanced. In that case, $\lambda$ is $e$-regular if and only if $\theta$ contains a zero.
\end{proposition}

Thus, we have already computed $M_e^{\gamma}(w)$ and $N_e^\gamma(w)$ in the previous section.

\begin{corollary}
    \label{co:newzero-count}
    For every odd integer $e>1$, every $e$-core $\gamma$, and every integer $w\ge 0$,
    \[
    M_e^\gamma(w)=\binom{w+\frac{e-1}{2}}{w}.
    \]
    and
    \[
    N_e^\gamma(w)=\binom{w+\frac{e-3}{2}}{w}.
    \] 
\end{corollary}

\begin{proof}
    This is a combination of Theorem~\ref{thm:maincountRouquier} and Proposition~\ref{prop:turkdbalancedmaximaloddsequence}.
\end{proof}

We now move on to prepare the battlefield for proving the other observed property in Example~\ref{ex:oddsequence-example-e5}.

\begin{lemma}\label{le:reappearance}
For every $e$-core $\gamma$ and integer $w\ge 0$,
\[
S_e^\gamma(w)\subseteq S_e^\gamma(w+2).
\]
\end{lemma}
\begin{proof} Given any partition $\lambda$, simply increasing its first part by $2e$ preserves the odd sequence as the parity and the $e$-residue of the right-most node of the first row are preserved. As this corresponds to adding two rim $e$-hooks, the $e$-core is preserved and the $e$-weight increases by two.
\end{proof}

So if an odd sequence appears in a block of $e$-weight $w$, it also appears in the blocks with the same $e$-core and $e$-weight $w+2, w+4, \dots$. This is not the case if we replace the even integers with odd ones. 
 
\begin{lemma}\label{le:parity}
For $e$-core $\gamma$ and integers $w, w'\ge 0$, if $\theta \in S_e^\gamma(w)\cap S_e^\gamma(w')$, then $w\equiv w' \pmod 2$.
\end{lemma}
\begin{proof}
Suppose $\lambda \vdash n$ has $e$-core $\gamma$ and odd sequence $\theta=(n_0,n_1,\ldots, n_{e-1})$. Then $\sum_{i=0}^{e-1}n_i$ is the number of odd parts in $\lambda$, which is congruent to $n$ modulo $2$. But $n=|\gamma|+ew$ and $e$ is odd so $ew \equiv w \pmod 2$. Thus
$$\sum_{i=0}^{e-1}n_i \equiv n =|\gamma|+ew \equiv |\gamma|+w \pmod 2.$$
The same is true if $w$ is replaced by $w'$; therefore, $w\equiv w' \pmod 2$.
\end{proof}

We are now ready to give the formulas for $T_e^\gamma(w)$ and $Z_e^\gamma(w)$. In the statements, we interpret $T_e^\gamma(w)$ and $Z_e^\gamma(w)$ as zero if $w<0$.

\begin{corollary}\label{cor:zero-recursion}
For every $e$-core $\gamma$ and integer $w\ge 0$, we have
\[
T_e^\gamma(w) = T_e^\gamma(w-2)+M_e^\gamma(w)
\]
and
\[
Z_e^\gamma(w)=Z_e^\gamma(w-2)+N_e^\gamma(w).
\]
\end{corollary}
 
\begin{proof}
By Lemma~\ref{le:reappearance}, every odd sequence $\theta\in S_e^\gamma(w-2)$ reappears in $S_e^\gamma(w)$. Conversely, if an odd sequence $\theta\in S_e^\gamma(w)$ occurs at a smaller $e$-weight $w'<w$, then, by Lemma~\ref{le:parity}, $w'=w-2k$ for some positive integer $k$. By applying Lemma~\ref{le:reappearance} $(k-1)$-times, we have that $\theta\in S_e^\gamma(w-2)$. Therefore, $S_e^\gamma(w-2)$ consists of the new odd sequences and odd sequences in $S_e^\gamma(w)$. Both formulae now follow.
\end{proof}

The recursive relations give rise to the generating functions for all four counting functions.

\begin{theorem}\label{th:generating functions}
Let $\gamma$ be an $e$-core partition. Then
\begin{align*}
    \sum_{w\ge 0} N_e^\gamma(w)x^w &= \frac{1}{(1-x)^{(e-1)/2}};\\
    \sum_{w\ge 0} M_e^\gamma(w)x^w &= \frac{1}{(1-x)^{(e+1)/2}};\\
    \sum_{w\ge 0} Z_e^\gamma(w)x^w &= \frac{1}{(1+x)(1-x)^{(e+1)/2}};\\
    \sum_{w\ge 0} T_e^\gamma(w)x^w &= \frac{1}{(1+x)(1-x)^{(e+3)/2}}.
\end{align*}
In particular, all four counting functions depend only on $e$ and $w$ and not on $\gamma$. 
\end{theorem}

\begin{proof}
    By Corollary~\ref{co:newzero-count}, we have
    \[
\sum_{w\ge 0} N_e^\gamma(w)x^w
=
\sum_{w\ge 0} \binom{w+\frac{e-3}{2}}{w}x^w
=
\frac{1}{(1-x)^{(e-1)/2}}
\]
and
\[
\sum_{w\ge 0} M_e^\gamma(w)x^w
=
\sum_{w\ge 0} \binom{w+\frac{e-1}{2}}{w}x^w
=
\frac{1}{(1-x)^{(e+1)/2}}.
\]
Corollary~\ref{cor:zero-recursion} shows that
\[
\sum_{w\ge 0} Z_e^\gamma(w)x^w
=
\frac{1}{1-x^2}\sum_{w\ge 0} N_e^\gamma(w)x^w
\]
and
\[
\sum_{w\ge 0} T_e^\gamma(w)x^w
=
\frac{1}{1-x^2}\sum_{w\ge 0} M_e^\gamma(w)x^w.
\]
We immediately obtain the last two generating functions from the first two.
\end{proof}

There are two interesting relations one can read from the generating functions of $Z_e^{\gamma}(w)$ and $T_e^\gamma(w)$. While there is a combinatorial explanation for the first one, the other one seems more mysterious.

\begin{corollary}\label{cor:shift-decomposition}
For every $e$-core $\gamma$ and integer $w\ge 0$, we have
\[
T_e^\gamma(w)=T_e^\gamma(w-1)+Z_e^\gamma(w).
\]
\end{corollary}

\begin{proof}
    This directly follows from Theorem~\ref{th:generating functions} after noting that
    \[
    \sum_{w\ge 0} T_e^\gamma(w)x^w = \frac{1}{1-x}\sum_{w\ge 0} Z_e^\gamma(w)x^w.
    \]
\end{proof}

Combinatorially, the equality in Corollary~\ref{cor:shift-decomposition} says that the odd sequences $S_e^\gamma(w)$ consisting of positive entries are in bijection with $S_e^\gamma(w-1)$ (which we interpret as the empty set if $w=0$). The bijection is given by the following map.

\begin{definition}
\label{definition:shiftoperationonoddsequence}
We define a shift map $\tau$ on odd sequences by 
$$
\tau((n_0,n_1,\dots,n_{e-1}))=(n_0+1,n_1+1,\dots,n_{e-1}+1).
$$
\end{definition}

Now recall the map $\psi$ from Section~\ref{se:RoCK} which appends to a partition $e$ parts equal to $1$. For each $i=0,1,\dots, e-1$, this function adds one odd part such that its right-most node has $e$-residue $i$. That is, $\tau(O(\lambda)) = O(\psi(\lambda))$. Since $\psi$ preserves $e$-cores and increases $e$-weight by one, the restriction
\[
\tau:S_e^\gamma(w-1) \rightarrow S_e^\gamma(w)
\]
is a well-defined injection. We can say a bit more.

\begin{proposition}\label{pr:shift}
Suppose that every entry of $\theta = (n_0, n_1,\dots,n_{e-1}) \in S_e^\gamma(w)$ is positive. Then
\[
(n_0-1, n_1 -1,\dots, n_{e-1}-1) \in S_e^\gamma(w-1).
\]
Thus the image of the shift map $\tau:S_e^\gamma(w-1) \rightarrow S_e^\gamma(w)$ is the subset of odd sequences in $S_e^\gamma(w)$ with all entries positive. 
\end{proposition}

\begin{proof}
    First let $w'=w_\theta(\gamma)$, i.e. $w'$  is the smallest $e$-weight where $\theta$ appears. Let $\mu$ be the maximal element in $\mathcal{E}_\theta(\gamma)$, so by Proposition \ref{prop:turkdbalancedmaximaloddsequence}, $\mu$ is $2$-balanced and $e$-singular. The latter shows that $w'>0$ and, by Lemma~\ref{le:add a column}, $\mu = \psi(\lambda)$ for some partition $\lambda$ (that is, $\mu$ has $e$ parts equal to one). As $\tau(O(\lambda)) = O(\mu)=\theta$, we see that $(n_0-1, n_1 -1,\dots, n_{e-1}-1) \in S_e^\gamma(w'-1)$.

    From Lemma~\ref{le:parity} we know $w \equiv w' \pmod 2$, so $w'=w-2k$ for some non-negative integer $k$. Applying Lemma~\ref{le:reappearance} $k$-times, we therefore have $(n_0-1, n_1 -1,\dots, n_{e-1}-1) \in S_e^\gamma(w'-1)\subseteq S_e^\gamma(w-1)$, as required.
\end{proof}

We conclude that the shift map indeed witnesses the identity $T_e^\gamma(w)=T_e^\gamma(w-1)+Z_e^\gamma(w)$. We can also use the shift map to informally describe the generating function of $T_e^\gamma(w)$; it admits a natural factorization reflecting the three propagation phenomena for odd sequences:
\[
\sum_{w\ge 0} T_e^\gamma(w)x^w
=
\underbrace{\frac{1}{1-x}}_{\substack{\text{shift factor:}\\\text{sequences without zeros}\\\text{are shifts of ones with zeros}}}
\cdot
\underbrace{\frac{1}{1-x^2}}_{\substack{\text{reappearance factor:}\\S_e^\gamma(w)\subseteq S_e^\gamma(w+2)}}
\cdot
\underbrace{\frac{1}{(1-x)^{(e-1)/2}}}_{\substack{\text{new-zero factor:}\\\sum_{w\ge0}N_e^\gamma(w)x^w}}.
\]
Each factor has a direct combinatorial meaning: the shift factor $\frac{1}{1-x}$ corresponds to sequences without zeros being precisely the shifts of sequences with zeros (Proposition~\ref{pr:shift}); the reappearance factor $\frac{1}{1-x^2}$ corresponds to the inclusion $S_e^\gamma(w)\subseteq S_e^\gamma(w+2)$ (Lemma~\ref{le:reappearance}); and the new-zero factor $\frac{1}{(1-x)^{(e-1)/2}}$ is the generating function for the genuinely new zero-containing sequences, whose count $N_e^\gamma(w)=\binom{w+\frac{e-3}{2}}{w}$ is given by Corollary~\ref{co:newzero-count}.

Finally, as promised, we provide one more formula for $T_e^\gamma(w)$, a Pascal triangle-like recurrence; see Table~\ref{tab:Te-table} for an example. There and in the statement $T_e(w)$ and $Z_e(w)$ denote the common value of $T_e^\gamma(w)$ and $Z_e^\gamma(w)$, respectively, now known to be independent of $\gamma$.

\begin{corollary}\label{cor:recurrence}
For every odd integer $e>3$ and integer $w\geq 0$,
\[
T_e(w)=T_e(w-1)+T_{e-2}(w).
\]
\end{corollary}
 
\begin{proof}
By Corollary~\ref{cor:shift-decomposition}
\[
T_e^\gamma(w)-T_e^\gamma(w-1)=Z_e(w).
\]
By Theorem~\ref{th:generating functions}, we have
\[
\sum_{w\ge 0} Z_e(w)x^w
=
\frac{1}{(1+x)(1-x)^{(e+1)/2}}
=
\frac{1}{(1+x)(1-x)^{((e-2)+3)/2}}
=
\sum_{w\ge 0} T_{e-2}(w)x^w.
\]
Hence
\[
T_e^\gamma(w)=T_e^\gamma(w-1)+T_{e-2}(w).
\]
\end{proof}
 
\begin{table}[ht]
\centering
\[
\begin{array}{c|cccccc}
 & w=0 & w=1 & w=2 & w=3 & w=4 & w=5 \\
\hline
e=3  & 1 & 2 & 4 & 6 & 9 & 12 \\
e=5  & 1 & 3 & 7 & 13 & 22 & 34 \\
e=7  & 1 & 4 & 11 & 24 & 46 & 80 \\
e=9  & 1 & 5 & 16 & 40 & 86 & 166 \\
e=11 & 1 & 6 & 22 & 62 & 148 & 314
\end{array}
\]
\caption{Values of $T_e(w)$ for odd $e$ and small $e$-weights $w$ illustrating the recurrence $T_e(w)=T_e(w-1)+T_{e-2}(w)$ from Corollary~\ref{cor:recurrence}.}
\label{tab:Te-table}
\end{table}

\section*{Acknowledgment}
The authors would like to thank Gunter Malle and Liron Speyer, whose questions were the catalysts for the main enumeration problem answered in this manuscript, Stacey Law for comments on an earlier version of the manuscript and Kai Meng Tan for pointing out that the strategy of counting certain partitions using crystal reflections is also used in \cite{ChuangMiyachiTanTilings17}. The second author is grateful to Bim Gustavsson and Stacey Law for many in-depth conversations while working on a connected project, and in particular, for their ideas that directly led to the proof of Lemma~\ref{le:add a column}.

The second author was supported by the LMS Early Career Fellowship ECF-2025-26 at the University of Birmingham and is currently funded by the FY2025 JSPS Postdoctoral Fellowship for Research in Japan (Short-term(PE)) PE25723 at the Okinawa Institute of Science and Technology.

\bibliographystyle{plain}
\bibliography{references}

@article {FayersPartitionModelsCrystal,
    AUTHOR = {Fayers, Matthew},
     TITLE = {Partition models for the crystal of the basic
              {$U_q(\widehat{\mathfrak{sl}}_n)$}-module},
   JOURNAL = {J. Algebraic Combin.},
  FJOURNAL = {Journal of Algebraic Combinatorics. An International Journal},
    VOLUME = {32},
      YEAR = {2010},
    NUMBER = {3},
     PAGES = {339--370},
      ISSN = {0925-9899,1572-9192},
   MRCLASS = {17B37 (05E10 17B10 17B67)},
  MRNUMBER = {2721055},
MRREVIEWER = {Darren\ Funk-Neubauer},
       DOI = {10.1007/s10801-010-0217-9},
       URL = {https://doi.org/10.1007/s10801-010-0217-9},
}

@article {FRT,
    AUTHOR = {Frame, J. S. and Robinson, G. de B. and Thrall, R. M.},
     TITLE = {The hook graphs of the symmetric groups},
   JOURNAL = {Canad. J. Math.},
  FJOURNAL = {Canadian Journal of Mathematics. Journal Canadien de
              Math\'ematiques},
    VOLUME = {6},
      YEAR = {1954},
     PAGES = {316--324},
      ISSN = {0008-414X,1496-4279},
   MRCLASS = {20.0X},
  MRNUMBER = {62127},
MRREVIEWER = {D.\ E.\ Littlewood},
       DOI = {10.4153/cjm-1954-030-1},
       URL = {https://doi.org/10.4153/cjm-1954-030-1},
}

@manual{GAP4,
    organization = "The GAP~Group",
    title        = "{GAP -- Groups, Algorithms, and Programming,
                    Version 4.15.1}",
    year         = 2025,
    url          = "\url{https://www.gap-system.org}",
}

@article {GKS,
    AUTHOR = {Garvan, Frank and Kim, Dongsu and Stanton, Dennis},
     TITLE = {Cranks and {$t$}-cores},
   JOURNAL = {Invent. Math.},
  FJOURNAL = {Inventiones Mathematicae},
    VOLUME = {101},
      YEAR = {1990},
    NUMBER = {1},
     PAGES = {1--17},
      ISSN = {0020-9910,1432-1297},
   MRCLASS = {11P83 (05A17 20C30 33D80)},
  MRNUMBER = {1055707},
MRREVIEWER = {George\ E.\ Andrews},
       DOI = {10.1007/BF01231493},
       URL = {https://doi.org/10.1007/BF01231493},
}

@misc{hemmerturek2026,
  author = {David J. Hemmer and Pavel Turek},
  title = {New columns in decomposition matrices of symmetric groups for every block},
  note = {Preprint},
  year = {2026}
}

@book {JamesKerber,
    AUTHOR = {James, Gordon and Kerber, Adalbert},
     TITLE = {The representation theory of the symmetric group},
    SERIES = {Encyclopedia of Mathematics and its Applications},
    VOLUME = {16},
      NOTE = {With a foreword by P. M. Cohn,
              With an introduction by Gilbert de B. Robinson},
 PUBLISHER = {Addison-Wesley Publishing Co., Reading, MA},
      YEAR = {1981},
     PAGES = {xxviii+510},
      ISBN = {0-201-13515-9},
   MRCLASS = {20-02 (20C30)},
  MRNUMBER = {644144},
MRREVIEWER = {A.\ O.\ Morris},
}

@book{kleshchev,
  author    = {A. Kleshchev},
  title     = {Linear and Projective Representations of Symmetric Groups},
  series    = {Cambridge Tracts in Mathematics},
  volume    = {163},
  publisher = {Cambridge University Press},
  year      = {2005}
}

@article{misramiwa,
  author    = {Kailash Misra and Tetsuji Miwa},
  title     = {Crystal base for the basic representation of {$U_q(\mathfrak{sl}(n))$}},
  journal   = {Comm. Math. Phys.},
  fjournal  = {Communications in Mathematical Physics},
  volume    = {134},
  year      = {1990},
  number    = {1},
  pages     = {79--88},
  issn      = {0010-3616, 1432-0916},
  mrclass   = {17B37},
  mrnumber  = {1079801},
  mrreviewer= {Andrew Pressley},
  url       = {http://projecteuclid.org/euclid.cmp/1104201614}
}

@article{turek2025mullineuxmapdbalancedpartitions,
      title={Mullineux map: $d$-balanced partitions and $d$-runner matrices}, 
      author={Pavel Turek},
      year={2025},
      journal={arXiv:2504.03864v2},
      primaryClass={math.CO},
      url={https://arxiv.org/abs/2504.03864}, 
}

@article {KleshchevBranchingIII96,
    AUTHOR = {Kleshchev, A. S.},
     TITLE = {Branching rules for modular representations of symmetric
              groups. {III}. {S}ome corollaries and a problem of
              {M}ullineux},
   JOURNAL = {J. London Math. Soc. (2)},
  FJOURNAL = {Journal of the London Mathematical Society. Second Series},
    VOLUME = {54},
      YEAR = {1996},
    NUMBER = {1},
     PAGES = {25--38},
      ISSN = {0024-6107,1469-7750},
   MRCLASS = {20C30 (20C20 20G05)},
  MRNUMBER = {1395065},
MRREVIEWER = {Jens\ C.\ Jantzen},
       DOI = {10.1112/jlms/54.1.25},
       URL = {https://doi.org/10.1112/jlms/54.1.25},
}

@article{ChuangMiyachiTanTilings17,
title = {Parallelotope tilings and {$q$}-decomposition numbers},
journal = {Advances in Mathematics},
volume = {321},
pages = {80--159},
year = {2017},
issn = {0001-8708},
doi = {https://doi.org/10.1016/j.aim.2017.09.024},
url = {https://www.sciencedirect.com/science/article/pii/S0001870816315262},
author = {Joseph Chuang and Hyohe Miyachi and Kai Meng Tan},
}

\end{document}